\documentclass[12pt]{amsart}
\usepackage{amscd,amssymb,amsthm,amsmath,amssymb,mathrsfs,enumerate,apptools}
\usepackage[matrix,arrow,curve]{xy}
\usepackage[margin=1.8cm]{geometry}

\usepackage{xcolor}
\usepackage{pdflscape}
\usepackage{graphicx}
\usepackage{array}
\usepackage[matrix,arrow,curve]{xy}
\usepackage{longtable}

\newcommand{\DR}{\mathbb{R}} 

\newcommand{\DZ}{\mathbb{Z}}

\newcommand{\DP}{\mathbb{P}}
\newcommand{\DA}{\mathbb{A}}

\newcommand{\MC}{\mathcal{C}}

\AtAppendix{\counterwithin{lemma}{section}}
\AtAppendix{\counterwithin{corollary}{section}}

\newtheorem{lemma}{Lemma}
\newtheorem{corollary}{Corollary}
\newtheorem{conjecture}{Conjecture}

\theoremstyle{definition}
\newtheorem{example}{Example}

\theoremstyle{remark}
\newtheorem{remark}{Remark}

\theoremstyle{maintheorem*}
\newtheorem*{theorem*}{Theorem}
\newtheorem*{maintheorem*}{Main Theorem}

\author{Elena Denisova}

\address{\emph{Elena Denisova}
		\newline
		\textnormal{University of Edinburgh,  Edinburgh, Scotland}
		\newline
		\textnormal{\texttt{s2223072@ed.ac.uk}}}

\title{K-stability of Fano 3-folds of Picard rank 3 and degree 20}

\thanks{Throughout this paper, all varieties are assumed to be projective and defined over~$\mathbb{C}$.}

\begin{document}

\begin{abstract}
We prove K-stability of smooth Fano 3-folds of Picard rank 3 and degree 20 that satisfy very explicit generality condition.
\end{abstract}

\maketitle

\tableofcontents

\section{Introduction}
\label{section:intro}
\noindent Let $S=\mathbb{P}^1\times\mathbb{P}^1$, let $C$ be a~smooth curve in $S$ of degree $(5,1)$,
and let $\eta\colon C\to\mathbb{P}^1$ be the~morphism induced by the projection $S\to\mathbb{P}^1$ to the first factor.
Then $\eta$ is a finite morphism of degree five, and we may assume that the~points $([1:0],[0:1])$ and $([0:1],[1:0])$ are among its ramifications points.
This assumption implies that the curve $C$ is given~by
$$
u\big(x^5+a_1x^{4}y+a_2x^{3}y^2+a_3x^{2}y^3\big)=v\big(y^5+b_1xy^4+b_2x^{2}y^3+b_3x^{3}y^2\big)
$$
for some $a_1$, $a_2$, $a_3$, $b_1$, $b_2$, $b_3$, where $([u:v],[x:y])$ are coordinates on $S$.
Note that the~ramification index of the~point  $([1:0],[0:1])$ can be computed as follows:
$$
\left\{\aligned
&2\ \text{if $a_3\ne 0$},\\
&3\ \text{if $a_3=0$ and $a_2\ne 0$},\\
&4\ \text{if $a_3=a_2=0$ and $a_1\ne 0$},\\
&5\ \text{if $a_3=a_2=a_1=0$}.
\endaligned
\right.
$$
Likewise, we can compute the~ramification index of the~point  $([0:1],[1:0])$.
We may assume~that
\begin{itemize}
\item $([1:0],[0:1])$ has the~largest ramification index among all ramifications points of $\eta$
\item the~ramification index of the~point $([0:1],[1:0])$ is the~second largest index.
\end{itemize}
If both these indices are $5$, then $a_1=a_2=a_3=b_1=b_2=b_3=0$, the morphism $\eta$ does not have other ramification points,
and the~equation of the~curve $C$ simplifies as
$$
ux^5=vy^5.
$$
In this case, we have $\mathrm{Aut}(S,C)\cong\mathbb{C}^\ast\rtimes\DZ/2\DZ$.
In all other cases, this group is finite \cite[Corollary~2.7]{CheltsovShramovPrzyjalkowski}.
\noindent Now, we consider embedding $S\hookrightarrow\mathbb{P}^1\times\mathbb{P}^2$ given by
$$
\big([u:v],[x:y]\big)\mapsto\big([u:v],[x^2:xy:y^2]\big),
$$
and identify $S$ and $C$ with their images in $\mathbb{P}^1\times\mathbb{P}^2$.
Let $\pi\colon X\to\mathbb{P}^1\times\mathbb{P}^2$ be the~blow up of the~curve~$C$. 
Then $X$ is a~smooth Fano threefold in the deformation family \textnumero~3.5 in the Mori--Mukai list
and every smooth member of this family can be obtained in this way. 
We know from \cite[Section 5.14]{Book}, that
\begin{itemize}
\item $X$ is K-stable if the numbers $a_1$, $a_2$, $a_3$, $b_1$, $b_2$, $b_3$ are general enough,
\item $X$ is K-polystable if $a_1=a_2=a_3=b_1=b_2=b_3=0$.
\end{itemize}
However, for some $a_1$, $a_2$, $a_3$, $b_1$, $b_2$, $b_3$, the threefold $X$ is not K-polystable.

\begin{example}
\label{example:not-K-polystable}
If $(a_1,a_2,a_3)=(0,0,0)\ne(b_1,b_2,b_3)$, then $X$ is not K-polystable \cite[Lemma 7.6]{Book}.
\end{example}
\noindent Note also that it follows from the~proof of \cite[Lemma~8.7]{CheltsovShramovPrzyjalkowski} that $\mathrm{Aut}(X)\cong\mathrm{Aut}(S,C)$.
In particular, we conclude the group $\mathrm{Aut}(X)$ is finite if and only if  $(a_1,a_2,a_3,b_1,b_2,b_3)\ne(0,0,0,0,0,0)$.
In this case, the threefold $X$ is K-polystable if and only if it is $K$-stable.
Moreover, we have

\begin{conjecture}[{\cite{Book}}]
\label{conjecture:main}
The Fano threefold $X$ is $K$-stable if and only if  $(a_1,a_2,a_3)\ne (0,0,0)$.
\end{conjecture}
\noindent  Geometrically, this conjecture says that the following two conditions are equivalent:
\begin{enumerate}
\item the threefold $X$ is K-stable,
\item the~morphism $\eta\colon C\to\mathbb{P}^1$ does not have ramification points of ramification index~five.
\end{enumerate}
The goal of this paper is to prove the following (slightly weaker) result:

\begin{theorem*}
\label{theorem:main}
If all ramification points of $\eta$ have ramification index~two, then $X$ is K-stable.
\end{theorem*}

\noindent Let  $\mathrm{pr}_1: \DP^1\times \DP^2 \to \DP^1$ be the projection to the first factor and $\phi_1=\mathrm{pr}_1\circ\pi$.
Then $\phi_1$ is a fibration into del Pezzo surfaces of degree four,
and Theorem  and Conjecture~\ref{conjecture:main} can be restated as follows:

\begin{maintheorem*}
\label{theorem:main-2}
If every singular fiber of $\phi_1$ has only singular points of type $\mathbb{A}_1$, then $X$ is K-stable.
\end{maintheorem*}
\begin{conjecture}
The Fano threefold $X$ is $K$-stable if and only if every singular fiber of $\phi_1$ has only singular points of type $\mathbb{A}_1$, $\mathbb{A}_2$ or
$\mathbb{A}_3$.
\end{conjecture}
\noindent
\noindent {\bf Acknowledgments:} I am grateful to my supervisor Professor Ivan Cheltsov for the introduction to the topic and his continuous support.

\section{The Proof}
\label{section:proof}
\noindent To prove {\bf Main Theorem}, we suppose that each singular fiber of the fibration $\phi_1$ has one or two singular~points of type $\mathbb{A}_1$. Note that this fiber is a del Pezzo surface of degree $4$ with Du Val singularities. We know (\cite{Fujita,Li}) that the Fano threefold $X$ is $K$-stable if and only if for every prime divisor $\mathbf{F}$ over $X$ we have
$$
\beta(\mathbf{F})=A_X(\mathbf{F})-S_X(\mathbf{F})>0
$$
 where $A_X(\mathbf{F})$ is the~log discrepancy of the~divisor $\mathbf{F}$, and
$$
S_X\big(\mathbf{F}\big)=\frac{1}{(-K_X)^3}\int\limits_0^{\infty}\mathrm{vol}\big(-K_X-u\mathbf{F}\big)du.
$$
To show this, we fix a prime divisor $\mathbf{F}$ over~$X$.
Then we set $Z=C_{X}(\mathbf{F})$.
If $Z$ is an irreducible~surface, then it follows from \cite{Fujita2016} that $\beta(\mathbf{F})>0$, see also \cite[Theorem 3.17]{Book}.
Therefore, we may assume that
\begin{itemize}
\item either  $Z$ is an irreducible curve in $X$,
\item or $Z$ is a point in $X$.
\end{itemize}
In both cases, we fix a point $O\in Z$. Let $\overline{T}$ be the fiber 
of $\phi_1$ which contains $O$. Then $\overline{T}$ is a del Pezzo surface with at most Du Val singularities. Set
$$\tau(\overline{T}) = \mathrm{sup}\Big\{u\in\DR_{>0}\big|\text{ the divisor }-K_X - u\overline{T} \text{ is pseudo-effective}\Big\}$$
For $u \in [0, \tau(\overline{T}) ]$ let $P(u)$ be the~positive part of the~Zariski decomposition of the~divisor $-K_X-u\overline{T}$,
and let $N(u)$ be its negative part. Then we have
$$
P(u)=\left\{\aligned
&-K_X-u\overline{T} \ \text{ if } u\in[0,1], \\
&-K_X-u\overline{T}-(u-1)\widetilde{S}\ \text{ if } u\in [1,2],
\endaligned
\right.
\text{ and }
N(u)= \left\{\aligned
&0\ \text{ if } u\in[0,1], \\
&(u-1)\widetilde{S}\ \text{ if } u\in[1,2],
\endaligned
\right.
$$
which gives 
$$S_{X}(\overline{T})=\frac{1}{20}\int\limits_{0}^{2}P(u)^3du=\frac{69}{80}<1$$ 
Now, for every prime divisor $F$ over the surface $\overline{T}$, we set
$$
S\big(W^{\overline{T}}_{\bullet,\bullet};F\big)=\frac{3}{(-K_X)^3}\int\limits_0^{\tau}\mathrm{ord}_F\big(N(u)\vert_{\overline{T}}\big)\big(P(u)\vert_{\overline{T}}\big)^2du+\frac{3}{(-K_X)^3}\int\limits_0^\tau\int\limits_0^{\infty}\mathrm{vol}\big(P(u)\big\vert_{\overline{T}}-vF\big)dvdu.
$$
Then, following \cite{AbbanZhuang,Book}, we let
$$
\delta_O\big(\overline{T},W^{\overline{T}}_{\bullet,\bullet}\big)=\inf_{\substack{F/\overline{T}\\O\in C_{\overline{T}}(F)}}\frac{A_T(F)}{S\big(W^{\overline{T}}_{\bullet,\bullet};F\big)},
$$
where the~infimum is taken by all prime divisors over the surface $\overline{T}$ whose center on $\overline{T}$ contains $O$.
Then it follows from \cite{AbbanZhuang,Book} that
\begin{equation*}
\label{equation:AZ}
\frac{A_X(\mathbf{F})}{S_X(\mathbf{F})}\geqslant\min\Bigg\{\frac{1}{S_X(\overline{T})},\delta_O\big(\overline{T},W^{\overline{T}}_{\bullet,\bullet}\big)\Bigg\}.
\end{equation*}
Therefore, if $\beta(\mathbf{F})\leqslant 0$,
then $\delta_O(\overline{T},W^{\overline{T}}_{\bullet,\bullet})\leqslant 1$.

\noindent Let's prove that $\delta_O(\overline{T},W^{\overline{T}}_{\bullet,\bullet}) > 1$. To estimate $\delta_O(T,W^{\overline{T}}_{\bullet,\bullet})$, we set $\overline{D}=P(u)\vert_{\overline{T}}$. We have 
$$
\overline{D}=\left\{\aligned
&-K_{\overline{T}} \ \text{ if } u\in[0,1], \\
&-K_{\overline{T}}-(u-1)\overline{C}_2\ \text{ if } u\in [1,2],
\endaligned
\right.$$
where $\overline{C}_2:=\widetilde{S}|_{\overline{T}}$. Then $\overline{D}$ is ample for $u\in[0,2)$, and
$$
\overline{D}^2=\left\{\aligned
& 4\ \text{ if } u\in[0,1], \\
& 5-u^2 \ \text{ if } u\in[1,2].
\endaligned
\right.
$$
We denote $\widetilde{S}$ to be the~proper transform on $X$ of the~surface $S$. By Lemma \cite[5.68]{Book} and Lemma\cite[5.69]{Book}  we have
\begin{lemma}
\label{lemma:3-5-S-P-delta}
If $O\in \widetilde{S}$ then $\delta_O(X) > 1$.
\end{lemma}
\begin{lemma}
\label{lemma:3-5-T-P-delta}
If $\overline{T}$ is smooth then $\delta_O(X)>1$.
\end{lemma}
\noindent Thus, to prove {\bf Main Theorem}, we may assume that $O\not\in\widetilde{S}$ and $\overline{T}$ is singular. Recall that
$$
\delta_O\big(\overline{T},\overline{D}\big)=\inf_{\substack{F/\overline{T}\\O\in C_{\overline{T}}(F)}}\frac{A_{\overline{T}}(F)}{S_{\overline{D}}\big(F\big)}\text{ where }\index{$S_D(F)$}
S_{\overline{D}}(F)=\frac{1}{\overline{D}^2}\int\limits_0^{\tau}\mathrm{vol}\big(\overline{D}-vF\big)dv$$
where $\tau=\tau(F)$ is the~pseudo-effective threshold of $F$ with respect to $\overline{D}$. Usually $\delta_O(\overline{T},-K_{\overline{T}})$  is denoted  by $\delta_O(\overline{T})$.
\\Note that since $O\not\in \widetilde{S}$ then for any divisor $F$ over $\overline{T}$ then  we get
\begin{align*}
S\big(W^{\overline{T}}_{\bullet,\bullet};F\big)&=\frac{3}{(-K_X)^3}\Bigg(\int_0^\tau\big(P(u)^{2}\cdot \overline{T}\big)\cdot\mathrm{ord}_{O}\Big(N(u)\big\vert_{\overline{T}}\Big)du+\int_0^\tau\int_0^\infty \mathrm{vol}\big(P(u)\big\vert_{\overline{T}}-vF\big)dvdu\Bigg)=\\
&= \frac{3}{20}\int_0^\tau\int_0^\infty \mathrm{vol}\big(P(u)\big\vert_{\overline{T}}-vF\big)dvdu=\\
&= \frac{3}{20}\Bigg(\int_0^1\int_0^\infty \mathrm{vol}\big(-K_{\overline{T}}-vF\big)dvdu+\int_1^2\int_0^\infty \mathrm{vol}\big(-K_{\overline{T}}-(u-1)\overline{C}_2-vF\big)dvdu\Bigg)=\\
&= \frac{3}{20}\Bigg(\int_0^\infty \mathrm{vol}\big(-K_{\overline{T}}-vF\big)dv+\int_0^\infty \mathrm{vol}\big(-K_T-(u-1)\overline{C}_2-vF\big)dv\Bigg)\le\\
&= \frac{3}{20}\Bigg(\int_0^\infty \mathrm{vol}\big(-K_{\overline{T}}-vF\big)dv+\int_0^\infty \mathrm{vol}\big(-K_{\overline{T}}-vF\big)dv\Bigg)=\\
&= \frac{3}{10}\Bigg(\int_0^\infty \mathrm{vol}\big(-K_{\overline{T}}-vF\big)dv\Bigg)=\frac{6}{5}\Bigg(\frac{1}{4}\int_0^\infty \mathrm{vol}\big(-K_{\overline{T}}-vF\big)dv\Bigg)=\\
&=\frac{6}{5}S_{\overline{T}}(F)\le \frac{6}{5}\cdot \frac{A_{\overline{T}}(F)}{\delta_O(\overline{T})}
\end{align*}
Thus, if $\delta_O(\overline{T})>6/5$, then $\delta_O(\overline{T},W^{\overline{T}}_{\bullet,\bullet})>1$. To estimate $\delta_O(\overline{T},W^{\overline{T}}_{\bullet,\bullet})$ in the case when $\delta_O(\overline{T})\le 6/5$, we define the following positive continuous function on $[1,2]$:
$$f(u):=
\left\{
\aligned
&\frac{15 - 3 u^2}{16 + 3 u - 9 u^2 + 2 u^3}\text{ if }u\in [1,a],\\
&\frac{15 - 3 u^2}{11 - u^3}\text{ if }u\in [a,2]
\endaligned
\right.$$
where $a$ is a root of $3u^3 - 9u^2 + 3u + 5$ on $[1,2]$. More precisely,  $a\in [1.355,1.356]$. In the appendix we prove that for each $O$ such that $\delta_O(\overline{T})\le \frac{6}{5}$ we have $\delta_{O}(\overline{T},\overline{D})\ge f(u)$ for every $u\in[1,2]$. So we obtain 
\begin{align*}
    S\big(W^{\overline{T}}_{\bullet,\bullet}&;F\big)=\frac{3}{(-K_X)^3}\int\limits_1^{2}\int\limits_0^\tau\mathrm{vol}\big(P(u)\big\vert_{\overline{T}}-vF\big)dvdu+\frac{3}{(-K_X)^3}\int\limits_0^{1}\int\limits_0^\tau\mathrm{vol}\big(P(u)\big\vert_{\overline{T}}-vF\big)dvdu\le\\
    &\le\frac{3}{20} \Bigg(\int \limits_1^{2}\frac{(5-u^2)}{\delta_O(\overline{T},\overline{D})}du\Bigg)A_{\overline{T}}(F)+\frac{3}{20}\cdot \frac{4A_{\overline{T}}(F)}{\delta_O(\overline{T})}\le\frac{3}{20} \Bigg(\int \limits_1^{2}\frac{(5-u^2)}{f(u)}du\Bigg)A_{\overline{T}}(F)+\frac{3}{5}A_{\overline{T}}(F)\le\\
    &\le \frac{3}{20} \Bigg(\int \limits_1^{1.356}(5-u^2)\frac{16 + 3 u - 9 u^2 + 2 u^3}{15 - 3 u^2} du+\int \limits_{1.355}^{2}(5-u^2)\frac{11 - u^3}{15 - 3 u^2} du\Bigg)A_{\overline{T}}(F)+\frac{3}{5}A_{\overline{T}}(F)\le\\
    &\le \frac{99}{100}A_{\overline{T}}(F)
\end{align*}
Thus $\frac{A_{\overline{T}}(F)}{S\big(W^{\overline{T}}_{\bullet,\bullet};F\big)}\ge\frac{100}{99}$  for every prime divisor $F$ over $\overline{T}$ whose support on $F$ contains $O$, so that $\delta_O(W^{\overline{T}},F)\ge \frac{100}{99}$,
which implies $\beta(\mathbf{F})>0$ and $X$ is $K$-stable. 
\begin{remark}
If $O$ were a singular point of type $\mathbb{A}_2$, this approach would not work, because as is shown in Appendix \ref{A2} there exists a curve $\overline{C}$ on $\overline{T}$ containing $O$  such that $\delta_{O}(\overline{T},\overline{D})= \frac{u^3 - 6u^2 + 19}{15-3u^2}$ so we get that 
   $$S(W^{\overline{T}};\overline{C}) \le \frac{3}{20} \Bigg(\int \limits_1^{2}\frac{(5-u^2)}{\delta_P(\overline{T},\overline{C})}du\Bigg)A_{\overline{T}}(\overline{C})+\frac{3}{5}A_{\overline{T}}(\overline{C}) =\frac{83}{80}A_{\overline{T}}(\overline{C}) $$
   so $\frac{A_{\overline{T}}(F)}{S\big(W^{\overline{T}}_{\bullet,\bullet};\overline{C}\big)}<1$ and we do not get a contradiction.
\end{remark}
\appendix
\section{Polarized $\delta$-invariant via  Kento Fujita’s formulas}
\noindent Let us use notations from Section 2. Recall that $\overline{T}$ is a Du Val del Pezzo surface, and the blow up $\pi$ induces a birational morphism $\upsilon:\overline{T}\to \DP^2$. We assume that $\overline{T}$ is singular so $\upsilon$ is a weighted blow up.  We have the following commutative diagram
$$\xymatrix{
&T\ar[dl]_{\sigma}\ar[dr]^{\eta}&\\
\overline{T}\ar[rr]^{\upsilon}& &\DP^2
}$$
Suppose that $u\in[1,2]$. Recall that $\overline{D}=-K_{\overline{T}}-(1-u)\overline{C}_2$. Observe that $\overline{C}_2$ is contained in the smooth locus of the surface $\overline{T}$. Let $C_2$ be the strict transform of the curve $\overline{C}_2$ on the surface $T$, set $D=-K_T-(1-u)C_2$. Note that $D=\sigma^*(\overline{D})$
so the divisor $D$ is big and nef for $u\in[1,2]$. Recall that
$$
\delta_O(\overline{T},\overline{D})=\inf_{\substack{F/\overline{T}\\ O\in C_{\overline{T}}(F)}}\frac{A_{\overline{T}}(F)}{S_D(F)}
$$
where the infimum is run over all prime divisor $F$ over $\overline{T}$ such that $O\in C_{\overline{T}}(F)$.  For every point $P\in T$, we also define
$$
\delta_P(T,D)=\inf_{\substack{E/T\\P\in C_T(E)}}\frac{A_T(E)}{S_D\big(E\big)}
$$
where the infimum is run over all prime divisor $E$ over $T$ such that $P\in C_T(E)$. Since $D=\sigma^*(\overline{D})$ and $K_T=\sigma^*(K_{\overline{T}}),$ we have
$$
\delta_O(\overline{T},\overline{D})=\inf_{P: O=\sigma(P)}\delta_P(T,D)
$$
So, to estimate $\delta_O(\overline{T},\overline{D})$ it is enough to estimate $\delta_P(T,D)$ for $P$ all points $P$ such that $\sigma(P)=O$.\\
Let $\MC$ be a smooth curve on $T$ containing $P$. 
Set
$$
\tau(\MC)=\mathrm{sup}\Big\{v\in\mathbb{R}_{\geqslant 0}\ \big\vert\ \text{the divisor  $-K_T-v\MC$ is pseudo-effective}\Big\}.
$$
For~$v\in[0,\tau]$, let $P(v)$ be the~positive part of the~Zariski decomposition of the~divisor $-K_T-\MC$,
and let $N(v)$ be its negative part. 
Then we set $$
S\big(W^{\MC}_{\bullet,\bullet};P\big)=\frac{2}{D^2}\int_0^{\tau(\MC)} h_D(v) dv,
$$
where
$$
h_D(v)=\big(P(v)\cdot \MC\big)\times\big(N(v)\cdot \MC\big)_P+\frac{\big(P(v)\cdot \MC\big)^2}{2}.
$$
It follows from \cite{AbbanZhuang,Book} that:
\begin{equation*}\label{estimation1}
    \delta_P(T,D)\geqslant\mathrm{min}\Bigg\{\frac{1}{S_D(\MC)},\frac{1}{S(W_{\bullet,\bullet}^{\MC},P)}\Bigg\}.
\end{equation*}
We will estimate $\delta_P(T,D)$ in the following using the notations above for a suitable choice of the curve $\MC$, $\tau(\MC)$, $P(v)$ and $N(v)$ later in special cases.
\\ A similar approach was taken in \cite{logbel4-24} and \cite{chelts3-22}.

\subsection{Polarized $\delta$-invariant on Del Pezzo surface of degree $4$ with $\DA_1$  singularity.} Suppose that $\overline{T}$ has one singular point and this point is a singular point of type $\mathbb{A}_1$. Then $\eta$ is a blow up of $\DP^2$ at points $P_1$, $P_2$, $P_3$ and $P_4$ in general position and a point $P_5$ which belongs to the exceptional divisor corresponding to $P_4$ and no other negative curve. Suppose $\mathbf{E}:=L_{14}\cup L_{24}\cup L_{24}\cup E_5$. By  \cite[Section 6.2]{Denisova} we have:
    $$\delta_P(T)=\left\{
\aligned
&1\text{ if }P\in E_4,\\
&6/5\text{ if }P\in \mathbf{E}\backslash E_4,\\
&4/3\text{ if }P \text{ belongs to two curves in }\{E_1,E_2,E_3,L_{12},L_{13},L_{23},L_{45},C_2\},\\
&18/13\text{ if }P\text{ belongs to exactly one curve in } \{E_1,E_2,E_3,L_{12},L_{13},L_{23},L_{45},C_2\}\backslash \mathbf{E},\\
&3/2,\text{ otherwise}
      \endaligned
      \right.
      $$
where $E_1$, $E_2$, $E_3$, $E_4$, $E_5$ are exceptional divisors corresponding to $P_1$, $P_2$, $P_3$, $P_4$, $P_5$ respectively, $C_{2}$ is a strict transform of a $(-1)$-curve coming from the conic on $\DP^2$, $L_{ij}$ are strict transforms of the lines passing through $P_i$ and $P_j$ for $(i,j)\in\{(1,2),(1,3),(1,4),(2,3),(2,4),(3,4)\}$ and $L_{45}$ a strict transform of a $(-1)$-curve coming from a line on $\DP^2$. The dual graph of $(-1)$ and $(-2)$-curves is given in the following picture:
    \begin{center}
      \includegraphics[width=11cm]{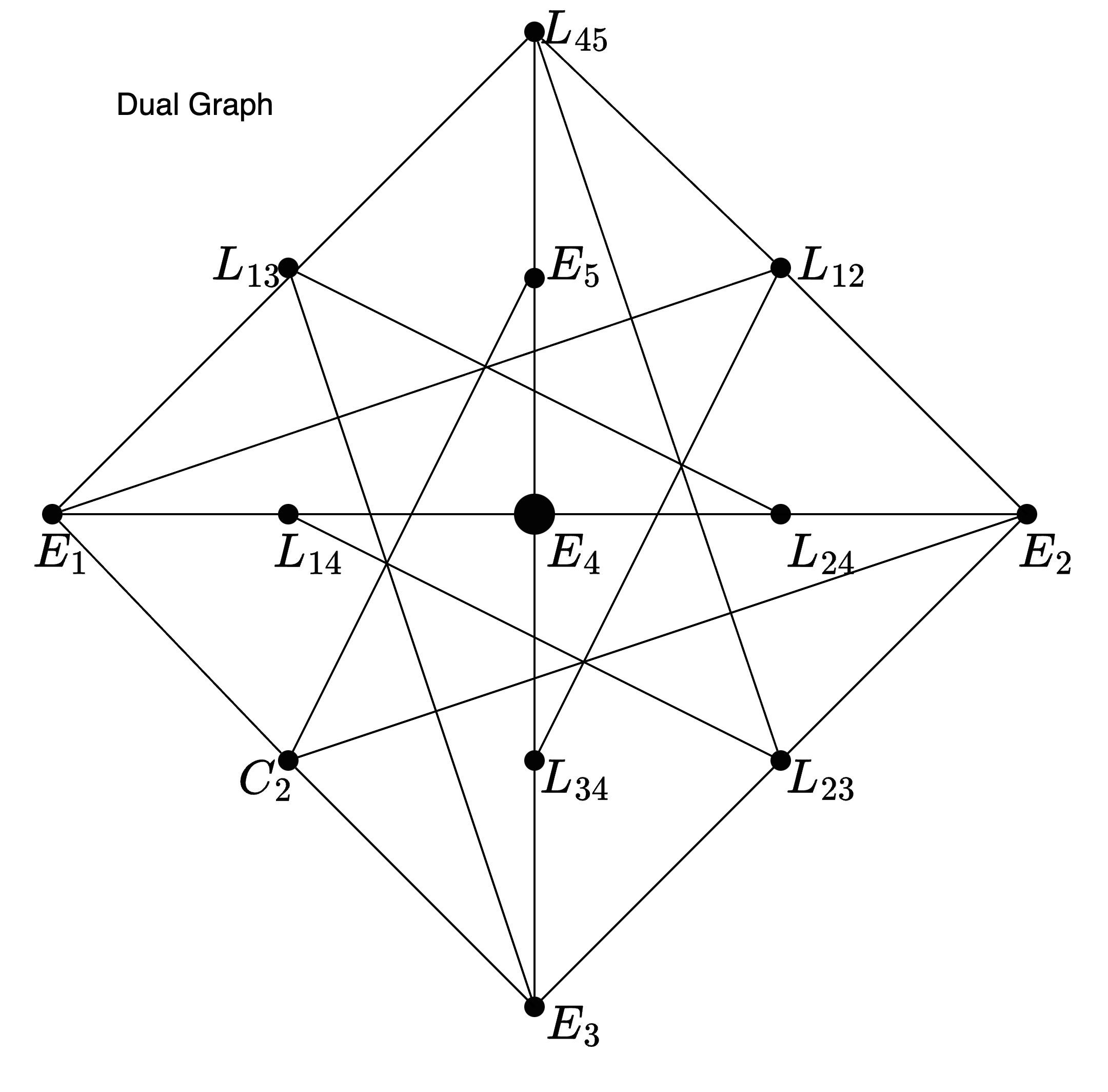}
   \end{center}
\begin{lemma}Suppose $P$ is a point on $T$ and $D=-K_{T}-(u-1)C_2$ with $D^2=5-u^2$ then
$$\delta_{P}(T,D)= 
\left\{
\aligned
&\frac{15 - 3 u^2}{16 + 3 u - 9 u^2 + 2 u^3} \ \text{ for }P\in E_4\backslash E_5  \text{ and }u\in[1,2]\\
&\frac{15 - 3 u^2}{11 - u^3}\text{ for }P\in E_5\backslash E_4\text{ and }u\in[1,2]\\
&\frac{15 - 3 u^2}{3 u^3 - 12 u^2 + 6 u + 13}\ \text{ for }P\in L_{14}\backslash (E_4\cup E_5 \cup E_1)\ \text{ and }u\in[1,2]\\
\endaligned
\right.$$
and 
$$\delta_{P}(T,D)\ge 
\left\{
\aligned
&\frac{15 - 3 u^2}{3 u^3 - 12 u^2 + 6 u + 13} \ \text{ for }P=E_4\cap E_5 \text{ and }u\in [1,a]\\
&\frac{15 - 3 u^2}{11 - u^3}\ \text{ for }P= E_4\cap E_5 \text{ and }u\in [a,2]\\
\endaligned
\right.$$
and 
$$\delta_{P}(T,D)\ge 
\left\{
\aligned
&\frac{15 - 3 u^2}{3 u^3 - 12 u^2 + 6 u + 13}\ \text{ for }P=L_{14}\cap E_1 \text{ and }u\in [1,b]\\
&\frac{2(15 - 3 u^2)}{19 - 2u^3}\ \text{ for }P=L_{14}\cap E_1 \text{ and }u\in [b,3/2]\\
&\frac{15 - 3 u^2}{3u^3 - 18u^2 + 27u - 4}\text{ for }P=L_{14}\cap E_1 \text{ and }u\in [3/2,2]
\endaligned
\right.$$
where $a$ is a root of $3u^3 - 9u^2 + 3u + 5$ on $[1,2]$,  $b$ is a root of $8u^3 - 24u^2 + 12u + 7$ on $[1,3/2]$. Note that  $a\in [1.355,1.356]$, $b\in [1.261,1.262]$.
\end{lemma}
\begin{proof}
{\bf Step 1.} Suppose $P\in E_4$. In this case we set $\MC= E_4$. Then $\tau(\MC)=3-u$. The Zariski Decomposition of the divisor $D-vE_4$ is given by:
$$P(v)=
\begin{cases}-K_{T}-(u-1)C_2-vE_4\text{ for }v\in[0,2-u]\\
-K_{T}-(u-1)C_2-vE_4-(u+v-2)E_5\text{ for }v\in[2-u,1]\\
-K_{T}-(u-1)C_2-vE_4-(u+v-2)E_5-(v-1)(L_{14}+L_{24}+L_{34})\text{ for }v\in[1,3-u]
\end{cases}$$
and 
$$N(v)=
\begin{cases}
0\text{ for }v\in[0,2-u]\\
(u+v-2)E_5\text{ for }v\in[2-u,1]\\
(u+v-2)E_5+(v-1)(L_{14}+L_{24}+L_{34})\text{ for }v\in[1,3-u]
\end{cases}$$
Moreover, 
$$P(v)^2=\begin{cases}5 -u^2 - 2v^2 \text{ for }v\in[0,2-u]\\
9+ 2uv - 4u - 4v - v^2\text{ for }v\in[2-u,1]\\
2(2 -v)(3 - u - v)\text{ for }v\in[1,3-u]
\end{cases}$$\text{ and }$$P(v)\cdot \MC=\begin{cases}
2v\text{ for }v\in[0,2-u]\\
2-u+v\text{ for }v\in[2-u,1]\\
5-u-2v\text{ for }v\in[1,3-u]
\end{cases}$$
Thus,
\begin{align*}
    S_D(\MC)=\frac{1}{5-u^2}\Big(\int_0^{2-u} 5 -u^2 - 2v^2 dv&+\int_{2-u}^1 9+ 2uv - 4u - 4v - v^2 dv+\\
    &+\int_{1}^{3-u} 2(2 -v)(3 - u - v) dv\Big)=\frac{16 + 3 u - 9 u^2 + 2 u^3}{15 - 3 u^2}
\end{align*}
Thus, $\delta_P(T,D)\le \frac{15 - 3 u^2}{16 + 3 u - 9 u^2 + 2 u^3}$ for $P\in E_4$. Note that we have:
\begin{itemize}
    \item if $P\in E_4\backslash (E_5\cup L_{14}\cup L_{24}\cup L_{34})$
$$h_D(v)=\begin{cases}2v^2\text{ for }v\in[0,2-u]\\
\frac{(2-u+v)^2}{2}\text{ for }v\in[2-u,1]\\
\frac{(5-u-2v)^2}{2}\text{ for }v\in[1,3-u]
\end{cases}$$
\item if $P=E_4\cap E_5$
$$h_D(v)=\begin{cases}2v^2\text{ for }v\in[0,2-u]\\
\frac{(2 - u + v) (u + 3 v - 2)}{2}\text{ for }v\in[2-u,1]\\
\frac{(u + 1) (5-u - 2 v)}{2}\text{ for }v\in[1,3-u]
\end{cases}$$
    \item if $P\in E_4\cap (L_{14}\cup L_{24}\cup L_{34})$
$$h_D(v)=\begin{cases}2v^2\text{ for }v\in[0,2-u]\\
\frac{(2-u+v)^2}{2}\text{ for }v\in[2-u,1]\\
\frac{(3 - u) (5-u - 2 v)}{2}\text{ for }v\in[1,3-u]
\end{cases}$$
\end{itemize}
So we have 
\begin{itemize}
    \item if $P\in E_4\backslash (E_5\cup L_{14}\cup L_{24}\cup L_{34})$ then
    \begin{align*}
    S_D(W_{\bullet,\bullet}^{\MC};P)=\frac{2}{5-u^2}\Big(\int_0^{2-u} 2v^2 dv&+\int_{2-u}^1 \frac{(2-u+v)^2}{2} dv+\int_{1}^{3-u} \frac{(5-u-2v)^2}{2} dv\Big)=\\
    &=\frac{9 + 6 u - 9 u^2 + 2 u^3}{15 - 3 u^2}\le \frac{16 + 3 u - 9 u^2 + 2 u^3}{15 - 3 u^2}
\end{align*}
   \item if $P=E_4\cap E_5$ then
    \begin{align*}
    S_D(W_{\bullet,\bullet}^{\MC};P)=\frac{2}{5-u^2}\Big(\int_0^{2-u} 2v^2 dv&+\int_{2-u}^1\frac{(2 - u + v) (u + 3 v - 2)}{2} dv+\\
    &+\int_{1}^{3-u} \frac{(u + 1) (5-u - 2 v)}{2} dv\Big)=\frac{11 - u^3}{15-3u^2}
\end{align*}
 \item if $P\in E_4\cap (L_{14}\cup L_{24}\cup L_{34})$ then
    \begin{align*}
    S_D(W_{\bullet,\bullet}^{\MC};P)=\frac{2}{5-u^2}\Big(\int_0^{2-u} 2v^2 dv&+\int_{2-u}^1 \frac{(2-u+v)^2}{2} dv+\\
    &+\int_{1}^{3-u} \frac{(3 - u) (5-u - 2 v)}{2}dv\Big)=\\
    &=\frac{13 + 3 u^3 - 12 u^2 + 6 u}{15 - 3u^2}\le \frac{16 + 3 u - 9 u^2 + 2 u^3}{15 - 3 u^2}
\end{align*}
\end{itemize}
We obtain that  
$$\delta_{P}(T,D)= \frac{15 - 3 u^2}{16 + 3 u - 9 u^2 + 2 u^3}\text{ for }P\in E_4\backslash E_5 \text{ and }u\in[1,2]$$
and 
$$\delta_{P}(T,D)\ge  
\left\{
\aligned
&\frac{15 - 3 u^2}{16 + 3 u - 9 u^2 + 2 u^3}\text{ for }P=E_4\cap E_5 \text{ and }u\in [1,a]\\
&\frac{15 - 3 u^2}{11 - u^3}\text{ for }P= E_4\cap E_5 \text{ and }u\in [a,2]
\endaligned
\right.$$
where $a$ is a root of $3u^3 - 9u^2 + 3u + 5$ on $[1,2]$. Note that $a\in [1.355,1.356]$.\\
{\bf Step 2.} Suppose  $P\in E_5$. In this case we set $\MC= E_5$. Then $\tau(\MC)=2$. The Zariski Decomposition of the divisor $D-vE_5$ is given by:
$$P(v)=
\begin{cases}-K_{T}-(u-1)C_2-vE_5-\frac{v}{2}E_4\text{ for }v\in[0,1]\\
-K_{T}-(u-1)C_2-vE_5-\frac{v}{2}E_4-(v-1)L_{45}\text{ for }v\in[1,u]\\
-K_{T}-(u-1)C_2-vE_5-\frac{v}{2}E_4-(v-1)L_{45}-(v-u)C_2\text{ for }v\in[u,2]
\end{cases}$$
and 
$$N(v)=
\begin{cases}\frac{v}{2}E_4\text{ for }v\in[0,1]\\
\frac{v}{2}E_4+(v-1)L_{45}\text{ for }v\in[1,u]\\
\frac{v}{2}E_4+(v-1)L_{45}+(v-u)C_2\text{ for }v\in[u,2]
\end{cases}$$
Moreover, 
$$P(v)^2=
\begin{cases} 5 - 4v + 2uv - u^2 - v^2/2 \text{ for }v\in[0,1]\\
6 - 6v + v^2/2 + 2uv - u^2\text{ for }v\in[1,u]\\
\frac{3(2 - v)^2}{2}\text{ for }v\in[u,2]
\end{cases}$$
and
$$P(v)\cdot \MC=\begin{cases}
2-u+v/2\text{ for }v\in[0,1]\\
3-u-v/2\text{ for }v\in[1,u]\\
3-3v/2\text{ for }v\in[u,2]
\end{cases}$$
Thus,
\begin{align*}
    S_D(\MC)=\frac{1}{5-u^2}\Big(\int_0^{1} 5 - 4v + 2uv - u^2 - v^2/2 dv&+\int_{1}^{u} 6 - 6v + v^2/2 + 2uv - u^2 dv+\\
    &+\int_{u}^{2} \frac{3(2 - v)^2}{3} dv\Big)=\frac{11 - u^3}{15 - 3 u^2}
\end{align*}
Thus, $\delta_P(T,D)\le \frac{15 - 3 u^2}{11 - u^3}$ for $P\in E_5$. Note that we have:
\begin{itemize}
\item if $P\in E_5\backslash (E_4\cup C_2\cup L_{45})$ then
$$h_D(v)=\begin{cases}
\frac{(2-u+v/2)^2}{2}\text{ for }v\in[0,1]\\
\frac{(3-u-v/2)^2}{2}\text{ for }v\in[1,u]\\
\frac{(3-3v/2)^2}{2}\text{ for }v\in[u,2]
\end{cases}$$
\item if $P=E_5\cap C_2$ then
$$h_D(v)=\begin{cases}
\frac{(2-u+v/2)^2}{2}\text{ for }v\in[0,1]\\
\frac{(3-u-v/2)^2}{2}\text{ for }v\in[1,u]\\
\frac{ 3 (2 - v) (6 -4 u + v)}{8}\text{ for }v\in[u,2]
\end{cases}$$
\item if $P= E_5\cap L_{45}$ then
$$h_D(v)=\begin{cases}
\frac{(2-u+v/2)^2}{2}\text{ for }v\in[0,1]\\
\frac{(6-2 u - v) (2-2 u + 3 v)}{8}\text{ for }v\in[1,u]\\
\frac{3(2 - v) (v + 2)}{8}\text{ for }v\in[u,2]
\end{cases}$$
\end{itemize}
So we have 
\begin{itemize}
    \item if $P\in  E_5\backslash (E_4\cup C_2\cup L_{45})$ then
    \begin{align*}
    S_D(W_{\bullet,\bullet}^{\MC};P)=\frac{2}{5-u^2}&\Big(\int_0^{1} \frac{(2-u+v/2)^2}{2} dv+\int_{1}^u \frac{(3-u-v/2)^2}{2} dv+\\
    &+\int_{u}^{2} \frac{(3-3v/2)^2}{2} dv\Big)=\frac{21 + 6 u - 18 u^2 + 5 u^3}{2(15 - 3 u^2)}\le \frac{11 - u^3}{15 - 3 u^2}
\end{align*}
   \item if $P=  E_5\cap C_2$ then
    \begin{align*}
    S_D(W_{\bullet,\bullet}^{\MC};P)=\frac{2}{5-u^2}&\Big(\int_0^{1} \frac{(2-u+v/2)^2}{2} dv+\int_{1}^u \frac{(3-u-v/2)^2}{2} dv+\\
    &+\int_{u}^{2} \frac{ 3 (2 - v) (6 -4 u + v)}{8} dv\Big)=\frac{45 - 30 u + 2 u^3}{2(15 - 3 u^2)}\le \frac{11 - u^3}{15 - 3 u^2}
\end{align*}
  \item if $P= E_5\cap L_{45}$ then
    \begin{align*}
    S_D(W_{\bullet,\bullet}^{\MC};P)=\frac{2}{5-u^2}&\Big(\int_0^{1} \frac{(2-u+v/2)^2}{2} dv+\int_{1}^u \frac{(6-2 u - v) (2-2 u + 3 v)}{8}dv+\\
    &+\frac{3(2 - v) (v + 2)}{8} dv\Big)=\frac{26 - 12 u^2 + 3 u^3}{2(15 - 3 u^2)}\le \frac{11 - u^3}{15 - 3 u^2}
\end{align*}
\end{itemize}
We obtain that $$\delta_P(T,D)= \frac{15 - 3 u^2}{11 - u^3}\text{ for }P\in E_5\backslash E_4\text{ and }u\in[1,2].$$\\
{\bf Step 3.1.} Suppose $P\in L_{14}\cup L_{24}\cup L_{34}$ and $u\in [1,3/2]$. In this case we set $\MC= L_{14}$. Then $\tau(\MC)=3-u$. Without loss of generality, we can assume that $P\in L_{14}$. The Zariski Decomposition of the divisor $D-vL_{14}$ is given by:
$$P(v)=
\begin{cases}D-vL_{14}-\frac{v}{2}E_4\text{ for }v\in[0,2-u]\\
D-vL_{14}-\frac{v}{2}E_4-(u+v-2)E_1\text{ for }v\in[2-u,1]\\
D-vL_{14}-\frac{v}{2}E_4-(u+v-2)E_1-(v-1)L_{23}\text{ for }v\in[1,4-2u]\\
D-vL_{14}-(u+v-2)(E_1+E_4)-(v-1)L_{23}-(2u+v-4)E_5\text{ for }v\in[4-2u,3-u]
\end{cases}$$
and 
$$N(v)=
\begin{cases}\frac{v}{2}E_4\text{ for }v\in[0,2-u]\\
\frac{v}{2}E_4+(u+v-2)E_1\text{ for }v\in[2-u,1]\\
\frac{v}{2}E_4+(u+v-2)E_1+(v-1)L_{23}\text{ for }v\in[1,4-2u]\\
(u+v-2)(E_1+E_4)+(v-1)L_{23}+(2u+v-4)E_5\text{ for }v\in[4-2u,3-u]
\end{cases}$$
Moreover
$$P(v)^2=
\begin{cases}5 - 2v - v^2/2 - u^2\text{ for }v\in[0,2-u]\\
9 - 4u - 6v + v^2/2 + 2uv\text{ for }v\in[2-u,1]\\
\frac{(v - 2)(3v + 4u - 10)}{2}\text{ for }v\in[1,4-2u]\\
2( u + v-3)^2\text{ for }v\in[4-2u,3-u]
\end{cases}$$
and
$$P(v)\cdot\MC=
\begin{cases}
v/2 + 1\text{ for }v\in[0,2-u]\\
3 - u - v/2\text{ for }v\in[2-u,1]\\
4 - u - 3v/2\text{ for }v\in[1,4-2u]\\
2(3 - u - v)\text{ for }v\in[4-2u,3-u]
\end{cases}$$
Thus,
\begin{align*}
    S_D(\MC)&=\frac{1}{5-u^2}\Big(\int_0^{2-u} 5 - 2v - v^2/2 - u^2 dv 
    +\int_{2-u}^{1} 9 - 4u - 6v + v^2/2 + 2uv dv+\\
    &+\int_{1}^{4-2u} \frac{(v - 2)(3v + 4u - 10)}{2} dv+\int_{4-2u}^{3-u} 2( u + v-3)^2 dv\Big)=\frac{3 u^3 - 12 u^2 + 6 u + 13}{15 - 3 u^2}
\end{align*}
Thus, $\delta_P(T,D)\le \frac{15 - 3 u^2}{3 u^3 - 12 u^2 + 6 u + 13}$ for $P\in L_{14}$. Note that we have:
\begin{itemize}
\item if $P\in L_{14}\backslash (E_4\cup E_1\cup L_{23}\cup E_5)$ then
$$h_D(v)=\begin{cases}
\frac{(v/2 + 1)^2}{2}\text{ for }v\in[0,2-u]\\
\frac{(3 - u - v/2)^2}{2}\text{ for }v\in[2-u,1]\\
\frac{(4 - u - 3v/2)^2}{2}\text{ for }v\in[1,4-2u]\\
2(3 - u - v)^2\text{ for }v\in[4-2u,3-u]
\end{cases}$$
\item if $P =  L_{14} \cap E_1$ then
$$h_D(v)=\begin{cases}
\frac{(v/2 + 1)^2}{2}\text{ for }v\in[0,2-u]\\
\frac{(6 - 2 u - v ) (2 u + 3 v - 2)}{8}\text{ for }v\in[2-u,1]\\
\frac{(8 -2 u - 3 v) (2 u + v)}{8}\text{ for }v\in[1,4-2u]\\
(3-u - v)\text{ for }v\in[4-2u,3-u]
\end{cases}$$
\item if $P = L_{14}\cap L_{23}$ then
$$h_D(v)=\begin{cases}
\frac{(v/2 + 1)^2}{2}\text{ for }v\in[0,2-u]\\
\frac{(3 - u - v/2)^2}{2}\text{ for }v\in[2-u,1]\\
\frac{(8 -2 u - 3 v) (4 -2 u + v)}{8}\text{ for }v\in[1,4-2u]\\
2 (2 - u) (3 - u - v)\text{ for }v\in[4-2u,3-u]
\end{cases}$$
\end{itemize}
So we have 
\begin{itemize}
    \item if $P\in  L_{14}\backslash (E_4\cup E_1\cup L_{23}\cup E_5)$ then
    \begin{align*}
    S_D(W_{\bullet,\bullet}^{\MC};P)=\frac{2}{5-u^2}&\Big(\int_0^{2-u} \frac{(v/2 + 1)^2}{2} dv+\int_{2-u}^1 \frac{(3 - u - v/2)^2}{2} dv+\\
    &+\int_{1}^{4-2u} \frac{(4 - u - 3v/2)^2}{2} dv+ \int_{4-2u}^{3-u} 2(3 - u - v)^2 dv\Big)=\\
    &=\frac{21-u^3 - 6u}{2(15 - 3 u^2)}\le \frac{3 u^3 - 12 u^2 + 6 u + 13}{15 - 3 u^2}
    \end{align*}
    \item if $P  = L_{14}\cap E_1$ then
    \begin{align*}
    S_D(W_{\bullet,\bullet}^{\MC}&;P)=\frac{2}{5-u^2}\Big(\int_0^{2-u} \frac{(v/2 + 1)^2}{2} dv+\int_{2-u}^1 \frac{(6 - 2 u - v ) (2 u + 3 v - 2)}{8} dv+\\
    &+\int_{1}^{4-2u} \frac{(8 -2 u - 3 v) (2 u + v)}{8} dv+ \int_{4-2u}^{3-u}(3-u - v) dv\Big)=\frac{19 - 2u^3}{2(15 - 3 u^2)}
    \end{align*}
        \item if $P  = L_{14}\cap L_{23}$ then
    \begin{align*}
    S_D(W_{\bullet,\bullet}^{\MC};P)=\frac{2}{5-u^2}&\Big(\int_0^{2-u} \frac{(v/2 + 1)^2}{2} dv+\int_{2-u}^1\frac{(3 - u - v/2)^2}{2} dv+\\
    &+\int_{1}^{4-2u} \frac{(8 -2 u - 3 v) (4 -2 u + v)}{8} dv+ \int_{4-2u}^{3-u} 2 (2 - u) (3 - u - v) dv\Big)=\\
    &=\frac{26 - 12u^2 + 3u^3}{2(15 - 3 u^2)}\le \frac{3 u^3 - 12 u^2 + 6 u + 13}{15 - 3 u^2}
    \end{align*}
\end{itemize}
We obtain that  
$$\delta_{P}(T,D)= 
\frac{15 - 3 u^2}{3 u^3 - 12 u^2 + 6 u + 13}\text{ for }P\in L_{14}\backslash (E_4\cup E_5 \cup E_1)
\text{ and }u\in[1,3/2]$$
and
$$\delta_{P}(T,D)\ge
\left\{
\aligned
&\frac{15 - 3 u^2}{3 u^3 - 12 u^2 + 6 u + 13}\text{ for }P=L_{14}\cap E_1 \text{ and }u\in [1,b]\\
&\frac{2(15 - 3 u^2)}{19 - 2u^3}\text{ for }P=L_{14}\cap E_1 \text{ and }u\in [b,3/2]
\endaligned
\right.$$
where $b$ is a root of $8u^3 - 24u^2 + 12u + 7$ on $[1,3/2]$. Note that $b\in [1.261,1.262]$.\\
{\bf Step 3.2.} Suppose $P\in L_{14}\cup L_{24}\cup L_{34}$ and $u\in [3/2,2]$.  In this case we set $\MC= L_{14}$. Then $\tau(\MC)=3-u$. Without loss of generality, we can assume that $P\in L_{14}$. The Zariski Decomposition of the divisor $D-vL_{14}$ is given by:
$$P(v)=
\begin{cases}D-vL_{14}-\frac{v}{2}E_4\text{ for }v\in[0,2-u]\\
D-vL_{14}-\frac{v}{2}E_4-(u+v-2)E_1\text{ for }v\in[2-u,4-2u]\\
D-vL_{14}-(u+v-2)(E_1+E_4)-(2u+v-4)E_5\text{ for }v\in[4-2u,1]\\
D-vL_{14}-(u+v-2)(E_1+E_4)-(v-1)L_{23}-(2u+v-4)E_5\text{ for }v\in[1, 3-u]
\end{cases}$$
and 
$$N(v)=
\begin{cases}
\frac{v}{2}E_4\text{ for }v\in[0,2-u]\\
\frac{v}{2}E_4+(u+v-2)E_1\text{ for }v\in[2-u,4-2u]\\
(u+v-2)(E_1+E_4)+(2u+v-4)E_5\text{ for }v\in[4-2u,1]\\
(u+v-2)(E_1+E_4)+(v-1)L_{23}+(2u+v-4)E_5\text{ for }v\in[1, 3-u]
\end{cases}$$
Moreover
$$P(v)^2=
\begin{cases}5 - 2v - v^2/2 - u^2\text{ for }v\in[0,2-u]\\
9 - 4u - 6v + v^2/2 + 2uv\text{ for }v\in[2-u,4-2u]\\
2u^2 + 4uv + v^2 - 12u - 10v + 17\text{ for }v\in[4-2u,1]\\
2(u + v-3)^2\text{ for }v\in[1,3-u]
\end{cases}$$
and
$$P(v)\cdot\MC=
\begin{cases}
1 + v/2\text{ for }v\in[0,2-u]\\
3 - u - v/2\text{ for }v\in[2-u,4-2u]\\
5 - 2u - v\text{ for }v\in[4-2u,1]\\
2(3 - u - v)\text{ for }v\in[1,3-u]
\end{cases}$$
Thus,
\begin{align*}
    S_D(\MC)&=\frac{1}{5-u^2}\Big(\int_0^{2-u} 5 - 2v - v^2/2 - u^2 dv 
    +\int_{2-u}^{4-2u} 9 - 4u - 6v + v^2/2 + 2uv dv+\\
    &+\int_{4-2u}^1 2u^2 + 4uv + v^2 - 12u - 10v + 17 dv+\int_{1}^{3-u} 2( u + v-3)^2 dv\Big)=\frac{3u^3 - 12u^2 + 6u + 13}{15 - 3 u^2}
\end{align*}
Thus, $\delta_P(T,D)\le \frac{15 - 3 u^2}{3u^3 - 12u^2 + 6u + 13}$ for $P\in L_{14}$. Note that we have:
\begin{itemize}
\item if $P\in L_{14}\backslash (E_4\cup E_1\cup L_{23}\cup E_5)$ then
$$h_D(v)=\begin{cases}
\frac{(1+ v/2)^2}{2}\text{ for }v\in[0,2-u]\\
\frac{(3 - u - v/2)^2}{2}\text{ for }v\in[2-u,4-2u]\\
\frac{(5 - 2u - v)^2}{2}\text{ for }v\in[4-2u,1]\\
2(3 - u - v)^2\text{ for }v\in[1,3-u]
\end{cases}$$
\item if $P = L_{14} \cap E_1$ then
$$h_D(v)=\begin{cases}
\frac{(1+ v/2)^2}{2}\text{ for }v\in[0,2-u]\\
\frac{ (6 -2 u - v ) (2 u + 3 v - 2)}{8}\text{ for }v\in[2-u,4-2u]\\
\frac{(v + 1) (5 -2 u - v)}{2}\text{ for }v\in[4-2u,1]\\
2 (3-u - v)\text{ for }v\in[1,3-u]
\end{cases}$$
\item if $P= L_{14}\cap L_{23}$ then
$$h_D(v)=\begin{cases}
\frac{(1+ v/2)^2}{2}\text{ for }v\in[0,2-u]\\
\frac{(3 - u - v/2)^2}{2}\text{ for }v\in[2-u,4-2u]\\
\frac{(5 - 2u - v)^2}{2}\text{ for }v\in[4-2u,1]\\
2 (2 - u) (3 -u - v)\text{ for }v\in[1,3-u]
\end{cases}$$
\end{itemize}
\begin{itemize}
    \item if $P\in  L_{14}\backslash (E_4\cup E_1\cup L_{23}\cup E_5)$ then
    \begin{align*}
    S_D(W_{\bullet,\bullet}^{\MC};P)=\frac{2}{5-u^2}&\Big(\int_0^{2-u} \frac{(v/2 + 1)^2}{2} dv+\int_{2-u}^{4-2u} \frac{(3 - u - v/2)^2}{2} dv+\\
    &+\int_{4-2u}^1 \frac{(5-2u-v)^2}{2} dv+ \int_{1}^{3-u} 2(3 - u - v)^2 dv\Big)=\\
    &=\frac{7u^3 - 36u^2 + 48u - 6}{2(15 - 3 u^2)}\le \frac{3 u^3 - 12 u^2 + 6 u + 13}{15 - 3 u^2}
    \end{align*}
    \item if $P = L_{14}\cap E_1$ then
    \begin{align*}
    S_D(W_{\bullet,\bullet}^{\MC}&;P)=\frac{2}{5-u^2}\Big(\int_0^{2-u} \frac{(v/2 + 1)^2}{2} dv+\int_{2-u}^{4-2u} \frac{ (6 -2 u - v ) (2 u + 3 v - 2)}{8} dv+\\
    &+\int_{4-2u}^1 \frac{(v + 1) (5 -2 u - v)}{2} dv+ \int_{1}^{3-u} 2 (3-u - v) dv\Big)=\frac{3u^3 - 18u^2 + 27u - 4}{15 - 3 u^2}
    \end{align*}
    \item if $P = L_{14}\cap L_{23}$ then
    \begin{align*}
    S_D(W_{\bullet,\bullet}^{\MC};P)=\frac{2}{5-u^2}&\Big(\int_0^{2-u} \frac{(v/2 + 1)^2}{2} dv+\int_{2-u}^{4-2u} \frac{(3 - u - v/2)^2}{2} dv+\\
    &+\int_{4-2u}^1 \frac{(5-2u-v)^2}{2} dv+ \int_{1}^{3-u} 2 (2 - u) (3 -u - v) dv\Big)=\\
    &=\frac{3u^3 - 12u^2 + 26}{2(15 - 3 u^2)}\le \frac{3 u^3 - 12 u^2 + 6 u + 13}{15 - 3 u^2}
    \end{align*}
\end{itemize}
We obtain that 
$$\delta_{P}(T,D)= \frac{15 - 3 u^2}{3 u^3 - 12 u^2 + 6 u + 13}\text{ for }P\in L_{14}\backslash (E_1\cup E_4 \cup E_5) \text{ and }u\in[3/2,2]$$
and
$$\delta_{P}(T,D)\ge \frac{15 - 3 u^2}{3u^3 - 18u^2 + 27u - 4}\text{ for }P=L_{14}\cap E_1 \text{ and }u\in [3/2,2]$$
\end{proof}
\begin{corollary}\label{corA1}
Let $P$ be a point in $T$ that is contained in $L_{12}\cup L_{24}\cup L_{34}\cup E_4 \cup E_5$ then
$$\delta_{P}(T,D)\ge
\left\{
\aligned
&\frac{15 - 3 u^2}{16 + 3 u - 9 u^2 + 2 u^3}\text{ for }u\in [1,a],\\
&\frac{15 - 3 u^2}{11 - u^3}\text{ for }u\in [a,2]
\endaligned
\right.$$
\end{corollary}
\begin{corollary}\label{corA1sing}
Suppose $O$ is a point on a del Pezzo surface $\overline{T}$ with $\mathbb{A}_1$ singularity and $\delta_O(T)\le \frac{6}{5}$ then
$$\delta_{O}(\overline{T},\overline{D})\ge
\left\{
\aligned
&\frac{15 - 3 u^2}{16 + 3 u - 9 u^2 + 2 u^3}\text{ for }u\in [1,a],\\
&\frac{15 - 3 u^2}{11 - u^3}\text{ for }u\in [a,2]
\endaligned
\right.$$
\end{corollary}
\subsection{Polarized $\delta$-invariant on Del Pezzo surface of degree $4$ with two $\DA_1$  singularities.}
Suppose that $\overline{T}$ has two singular points and these points are singular point of type $\mathbb{A}_1$. Then $\eta$ is a blow up of $\DP^2$ at points  $P_1$, $P_2$, and $P_4$ in general position and after that blowing up a point $P_3$ which belongs to the exceptional divisor corresponding to $P_2$ and a point $P_5$ which belongs to the exceptional divisor corresponding to $P_4$ and no other negative curve. By  \cite[Section 6.2]{Denisova} we have:
$$\delta_P(T)=
\left\{
\aligned
&1\text{ if }P\in (E_2\cup E_4\cup L_{24}),\\
& 6/5\text{ if }P\in (E_3\cup E_5\cup L_{12}\cup L_{14})\backslash (E_2\cup E_4),\\
 &4/3\text{ if }P\in (C_2\cap E_1)\cup (L_{23}\cap L_{45}),\\
 &18/13\text{ if }P\in (C_2\cup E_1\cup L_{23}\cup L_{45})\backslash\big((C_2\cap E_1)\cup (L_{23}\cap L_{45})\cup  (E_3\cup E_5\cup L_{12}\cup L_{14})\big),\\
 &3/2,\text{ otherwise }
\endaligned
\right.$$
where $E_1$, $E_2$, $E_3$, $E_4$, $E_5$ are exceptional divisors corresponding to $P_1$, $P_2$, $P_3$, $P_4$, $P_5$ respectively, $C_{2}$ is a strict transform of a $(-1)$-curve coming from the conic on $\DP^2$, $L_{ij}$ are strict transforms of the lines passing through $P_i$ and $P_j$ for $(i,j)\in\{(1,2),(1,4)\}$ and $L_{45}$ and $L_{23}$ are strict transforms of a $(-1)$-curve coming from  lines passing through $P_2$ and $P_4$ respectively on $\DP^2$. The dual graph of $(-1)$ and $(-2)$-curves is given in the following picture:
       \begin{center}
      \includegraphics[width=13cm]{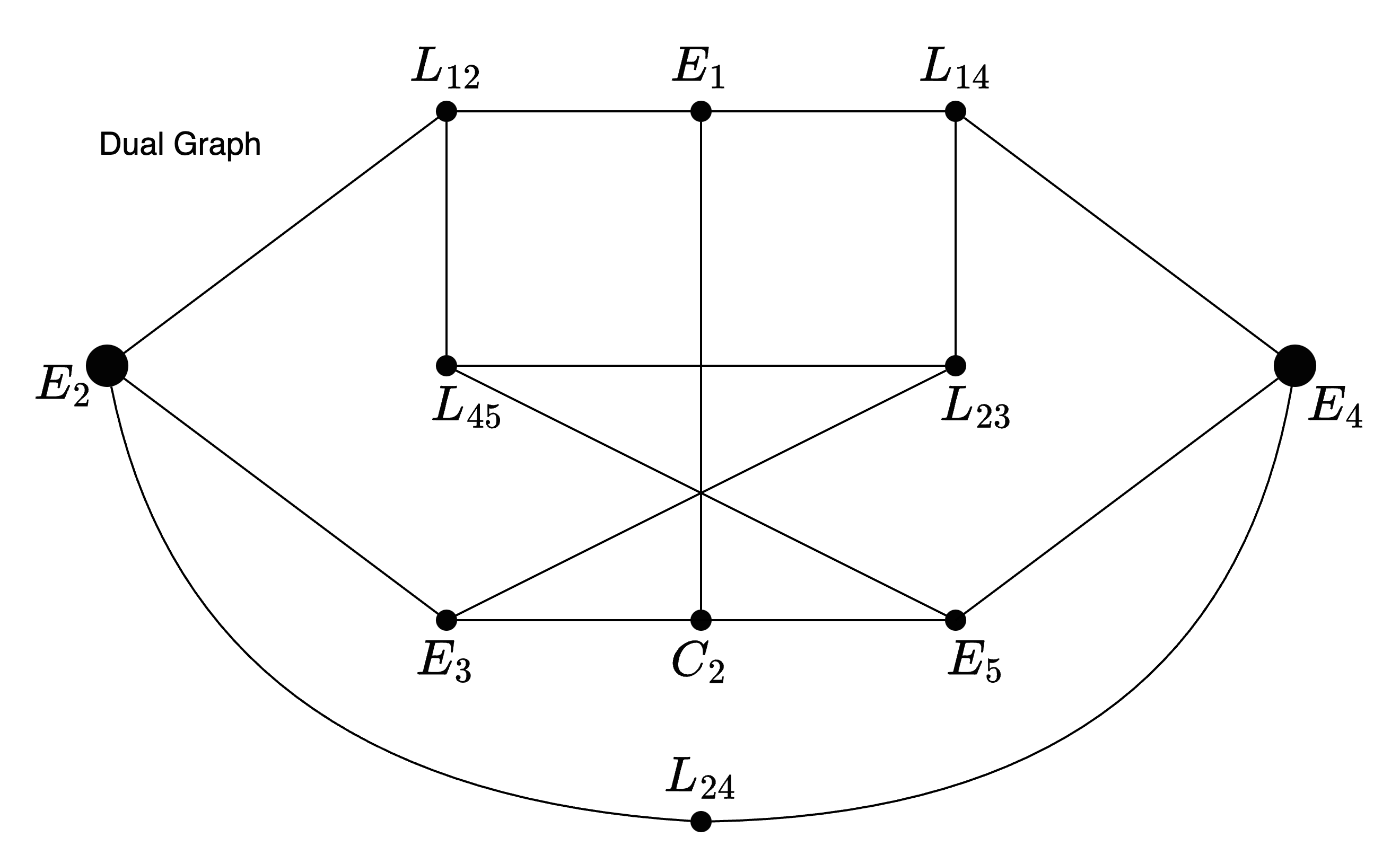}
   \end{center}
\begin{lemma}Suppose $P$ is a point on $T$ and $D=-K_{T}-(u-1)C_2$ with $D^2=5-u^2$ then
$$\delta_{P}(T,D)= 
\left\{
\aligned
&\frac{15 - 3 u^2}{16 + 3 u - 9 u^2 + 2 u^3}\text{ for }P\in (E_2\cup E_4)\backslash (E_3\cup E_5) \text{ and }u\in[1,2],\\
&\frac{15 - 3 u^2}{11 - u^3}\text{ for }P\in (E_3\cup E_5)\backslash (E_2\cup E_4)\text{ and }u\in[1,2],\\
&\frac{15 - 3 u^2}{u^3 - 6u^2 + 6u + 5}\text{ for }P\in L_{24}\backslash (E_2\cup E_4)\text{ and }u\in[1,2],\\
&\frac{15 - 3 u^2}{3 u^3 - 12 u^2 + 6 u + 13}\text{ for }P\in L_{14}\backslash (E_4\cup E_5 \cup E_1)
\text{ and }u\in[1,2],
\endaligned
\right.$$
and 
$$\delta_{P}(T,D)\ge
\left\{
\aligned
&\frac{15 - 3 u^2}{16 + 3 u - 9 u^2 + 2 u^3}\text{ for }P\in \{E_2\cap E_3 ,E_4 \cap E_5\} \text{ and }u\in [1,a],\\
&\frac{15 - 3 u^2}{11 - u^3}\text{ for }P\in \{E_2\cap E_3 ,E_4 \cap E_5\} \text{ and }u\in [a,2]
\endaligned
\right.$$
and 
$$\delta_{P}(T,D)\ge
\left\{
\aligned
&\frac{15 - 3 u^2}{3 u^3 - 12 u^2 + 6 u + 13}\text{ for }P=L_{14}\cap E_1 \text{ and }u\in [1,b],\\
&\frac{2(15 - 3 u^2)}{19 - 2u^3}\text{ for }P=L_{14}\cap E_1 \text{ and }u\in [b,3/2],\\
&\frac{15 - 3 u^2}{3u^3 - 18u^2 + 27u - 4}\text{ for }P=L_{14}\cap E_1 \text{ and }u\in [3/2,2],
\endaligned
\right.$$
where $a$ is a root of $3u^3 - 9u^2 + 3u + 5$ on $[1,2]$,  $b$ is a root of $8u^3 - 24u^2 + 12u + 7$ on $[1,3/2]$. Note that  $a\in [1.355,1.356]$, $b\in [1.261,1.262]$.
\end{lemma}
\begin{proof}
{\bf Step 1.} Suppose $P\in E_2\cup E_4$.  Without loss of generality we can assume that $P\in E_4$. In this case we set $\MC= E_4$. Then $\tau(\MC)=3-u$. The Zariski Decomposition of the divisor $D-vE_4$ is given by:
$$P(v)=
\begin{cases}-K_{T}-(u-1)C_2-vE_4\text{ for }v\in[0,2-u]\\
-K_{T}-(u-1)C_2-vE_4-(u+v-2)E_5\text{ for }v\in[2-u,1]\\
-K_{T}-(u-1)C_2-vE_4-(u+v-2)E_5-(v-1)(L_{14}+2L_{24}+E_2)\text{ for }v\in[1,3-u]
\end{cases}$$
and 
$$N(v)=
\begin{cases}
0\text{ for }v\in[0,2-u]\\
(u+v-2)E_5\text{ for }v\in[2-u,1]\\
(u+v-2)E_5+(v-1)(L_{14}+2L_{24}+E_2)\text{ for }v\in[1,3-u]
\end{cases}$$
Moreover, 
$$P(v)^2=\begin{cases}5 -u^2 - 2v^2 \text{ for }v\in[0,2-u]\\
9+ 2uv - 4u - 4v - v^2\text{ for }v\in[2-u,1]\\
2(2 -v)(3 - u - v)\text{ for }v\in[1,3-u]
\end{cases}$$\text{ and }$$P(v)\cdot \MC=\begin{cases}
2v\text{ for }v\in[0,2-u]\\
2-u+v\text{ for }v\in[2-u,1]\\
5-u-2v\text{ for }v\in[1,3-u]
\end{cases}$$
Thus,
\begin{align*}
    S_D(\MC)=\frac{1}{5-u^2}\Big(\int_0^{2-u} 5 -u^2 - 2v^2 dv&+\int_{2-u}^1 9+ 2uv - 4u - 4v - v^2 dv+\\
    &+\int_{1}^{3-u} 2(2 -v)(3 - u - v) dv\Big)=\frac{16 + 3 u - 9 u^2 + 2 u^3}{15 - 3 u^2}
\end{align*}
Thus, $\delta_P(T,D)\le \frac{15 - 3 u^2}{16 + 3 u - 9 u^2 + 2 u^3}$ for $P\in E_4$. Note that we have:
\begin{itemize}
    \item if $P\in E_4\backslash (E_5\cup L_{14}\cup L_{24}\cup L_{34})$
$$h_D(v)=\begin{cases}2v^2\text{ for }v\in[0,2-u]\\
\frac{(2-u+v)^2}{2}\text{ for }v\in[2-u,1]\\
\frac{(5-u-2v)^2}{2}\text{ for }v\in[1,3-u]
\end{cases}$$
\item if $P=E_4\cap E_5$
$$h_D(v)=\begin{cases}2v^2\text{ for }v\in[0,2-u]\\
\frac{(2 - u + v) (u + 3 v - 2)}{2}\text{ for }v\in[2-u,1]\\
\frac{(u + 1) (5-u - 2 v)}{2}\text{ for }v\in[1,3-u]
\end{cases}$$
    \item if $P\in E_4\cap (L_{14}\cup L_{24})$
$$h_D(v)\le \begin{cases}2v^2\text{ for }v\in[0,2-u]\\
\frac{(2-u+v)^2}{2}\text{ for }v\in[2-u,1]\\
\frac{(5-u - 2 v) (1 -u + 2 v)}{2}\text{ for }v\in[1,3-u]
\end{cases}$$
\end{itemize}
So we have 
\begin{itemize}
    \item if $P\in E_4\backslash (E_5\cup L_{14}\cup L_{24}\cup L_{34})$ then
    \begin{align*}
    S_D(W_{\bullet,\bullet}^{\MC};P)=\frac{2}{5-u^2}\Big(\int_0^{2-u} 2v^2 dv&+\int_{2-u}^1 \frac{(2-u+v)^2}{2} dv+\int_{1}^{3-u} \frac{(5-u-2v)^2}{2} dv\Big)=\\
    &=\frac{9 + 6 u - 9 u^2 + 2 u^3}{15 - 3 u^2}\le \frac{16 + 3 u - 9 u^2 + 2 u^3}{15 - 3 u^2}
\end{align*}
   \item if $P=E_4\cap E_5$ then
    \begin{align*}
    S_D(W_{\bullet,\bullet}^{\MC};P)=\frac{2}{5-u^2}\Big(\int_0^{2-u} 2v^2 dv&+\int_{2-u}^1\frac{(2 - u + v) (u + 3 v - 2)}{2} dv+\\
    &+\int_{1}^{3-u} \frac{(u + 1) (5-u - 2 v)}{2} dv\Big)=\frac{11 - u^3}{15-3u^2}
\end{align*}
 \item if $P\in E_4\cap (L_{14}\cup L_{24})$ then
    \begin{align*}
    S_D(W_{\bullet,\bullet}^{\MC};P)=\frac{2}{5-u^2}\Big(\int_0^{2-u} 2v^2 dv&+\int_{2-u}^1 \frac{(2-u+v)^2}{2} dv+\int_{1}^{3-u} \frac{(5-u - 2 v) (1 -u + 2 v)}{2}dv\Big)=\\
    &=\frac{2u^3 - 6u^2 + 8}{15 - 3u^2}\le \frac{16 + 3 u - 9 u^2 + 2 u^3}{15 - 3 u^2}
\end{align*}
\end{itemize}
We obtain that  
$$\delta_{P}(T,D)= \frac{15 - 3 u^2}{16 + 3 u - 9 u^2 + 2 u^3}\text{ for }P\in E_4\backslash E_5 \text{ and }u\in[1,2]$$
and 
$$\delta_{P}(T,D)\ge 
\left\{
\aligned
&\frac{15 - 3 u^2}{16 + 3 u - 9 u^2 + 2 u^3}\text{ for }P=E_4\cap E_5 \text{ and }u\in [1,a]\\
&\frac{15 - 3 u^2}{11 - u^3}\text{ for }P= E_4\cap E_5 \text{ and }u\in [a,2]
\endaligned
\right.$$
where $a$ is a root of $3u^3 - 9u^2 + 3u + 5$ on $[1,2]$. Note that $a\in [1.355,1.356]$.\\
{\bf Step 2.} Suppose $P\in E_3\cup E_5$. In this case we set $\MC= E_5$. Then $\tau(\MC)=2$. Without loss of generality we can assume that $P\in E_5$. The Zariski Decomposition of the divisor $D-vE_5$ is given by:
$$P(v)=
\begin{cases}-K_{T}-(u-1)C_2-vE_5-\frac{v}{2}E_4\text{ for }v\in[0,1]\\
-K_{T}-(u-1)C_2-vE_5-\frac{v}{2}E_4-(v-1)L_{45}\text{ for }v\in[1,u]\\
-K_{T}-(u-1)C_2-vE_5-\frac{v}{2}E_4-(v-1)L_{45}-(v-u)C_2\text{ for }v\in[u,2]
\end{cases}$$
and 
$$N(v)=
\begin{cases}\frac{v}{2}E_4\text{ for }v\in[0,1]\\
\frac{v}{2}E_4+(v-1)L_{45}\text{ for }v\in[1,u]\\
\frac{v}{2}E_4+(v-1)L_{45}+(v-u)C_2\text{ for }v\in[u,2]
\end{cases}$$
Moreover, 
$$P(v)^2=
\begin{cases} 5 - 4v + 2uv - u^2 - v^2/2 \text{ for }v\in[0,1]\\
6 - 6v + v^2/2 + 2uv - u^2\text{ for }v\in[1,u]\\
\frac{3(2 - v)^2}{2}\text{ for }v\in[u,2]
\end{cases}$$
and
$$P(v)\cdot \MC=\begin{cases}
2-u+v/2\text{ for }v\in[0,1]\\
3-u-v/2\text{ for }v\in[1,u]\\
3-3v/2\text{ for }v\in[u,2]
\end{cases}$$
Thus,
\begin{align*}
    S_D(\MC)=\frac{1}{5-u^2}\Big(\int_0^{1} 5 - 4v + 2uv - u^2 - v^2/2 dv&+\int_{1}^{u} 6 - 6v + v^2/2 + 2uv - u^2 dv+\\
    &+\int_{u}^{2} \frac{3(2 - v)^2}{3} dv\Big)=\frac{11 - u^3}{15 - 3 u^2}
\end{align*}
Thus, $\delta_P(T,D)\le \frac{15 - 3 u^2}{11 - u^3}$ for $P\in E_5$. Note that we have:
\begin{itemize}
\item if $P\in E_5\backslash (E_4\cup C_2\cup L_{45})$ then
$$h_D(v)=\begin{cases}
\frac{(2-u+v/2)^2}{2}\text{ for }v\in[0,1]\\
\frac{(3-u-v/2)^2}{2}\text{ for }v\in[1,u]\\
\frac{(3-3v/2)^2}{2}\text{ for }v\in[u,2]
\end{cases}$$
\item if $P=E_5\cap C_2$ then
$$h_D(v)=\begin{cases}
\frac{(2-u+v/2)^2}{2}\text{ for }v\in[0,1]\\
\frac{(3-u-v/2)^2}{2}\text{ for }v\in[1,u]\\
\frac{ 3 (2 - v) (6 -4 u + v)}{8}\text{ for }v\in[u,2]
\end{cases}$$
\item if $P= E_5\cap L_{45}$ then
$$h_D(v)=\begin{cases}
\frac{(2-u+v/2)^2}{2}\text{ for }v\in[0,1]\\
\frac{(6-2 u - v) (2-2 u + 3 v)}{8}\text{ for }v\in[1,u]\\
\frac{3(2 - v) (v + 2)}{8}\text{ for }v\in[u,2]
\end{cases}$$
\end{itemize}
So we have 
\begin{itemize}
    \item if $P\in  E_5\backslash (E_4\cup C_2\cup L_{45})$ then
    \begin{align*}
    S_D(W_{\bullet,\bullet}^{\MC};P)=\frac{2}{5-u^2}&\Big(\int_0^{1} \frac{(2-u+v/2)^2}{2} dv+\int_{1}^u \frac{(3-u-v/2)^2}{2} dv+\\
    &+\int_{u}^{2} \frac{(3-3v/2)^2}{2} dv\Big)=\frac{21 + 6 u - 18 u^2 + 5 u^3}{2(15 - 3 u^2)}\le \frac{11 - u^3}{15 - 3 u^2}
\end{align*}
   \item if $P=  E_5\cap C_2$ then
    \begin{align*}
    S_D(W_{\bullet,\bullet}^{\MC};P)=\frac{2}{5-u^2}&\Big(\int_0^{1} \frac{(2-u+v/2)^2}{2} dv+\int_{1}^u \frac{(3-u-v/2)^2}{2} dv+\\
    &+\int_{u}^{2} \frac{ 3 (2 - v) (6 -4 u + v)}{8} dv\Big)=\frac{45 - 30 u + 2 u^3}{2(15 - 3 u^2)}\le \frac{11 - u^3}{15 - 3 u^2}
\end{align*}
  \item if $P= E_5\cap L_{45}$ then
    \begin{align*}
    S_D(W_{\bullet,\bullet}^{\MC};P)=\frac{2}{5-u^2}&\Big(\int_0^{1} \frac{(2-u+v/2)^2}{2} dv+\int_{1}^u \frac{(6-2 u - v) (2-2 u + 3 v)}{8}dv+\\
    &+\frac{3(2 - v) (v + 2)}{8} dv\Big)=\frac{26 - 12 u^2 + 3 u^3}{2(15 - 3 u^2)}\le \frac{11 - u^3}{15 - 3 u^2}
\end{align*}
\end{itemize}
We obtain that $$\delta_P(T,D)= \frac{15 - 3 u^2}{11 - u^3}\text{ for }P\in (E_3\cup E_5)\backslash (E_2\cup E_4)\text{ and }u\in[1,2].$$\\
{\bf Step 3.} Suppose  $P\in L_{24}$. In this case we set $\MC= L_{24}$. Then $\tau(\MC)=3-u$. The Zariski Decomposition of the divisor $D-vL_{24}$ is given by:
$$P(v)=
\begin{cases}D-vL_{24}-\frac{v}{2}(E_2+E_4)\text{ for }v\in[0,4-2u]\\
D-vL_{24}-(u+v-2)(E_2+E_4)-(2u+v-4)(E_3+E_5)\text{ for }v\in[4-2u,3-u]
\end{cases}$$
and 
$$N(v)=
\begin{cases}\frac{v}{2}(E_2+E_4)\text{ for }v\in[0,4-2u]\\
(u+v-2)(E_2+E_4)+(2u+v-4)(E_3+E_5)\text{ for }v\in[4-2u, 3-u]
\end{cases}$$
Moreover,
$$P(v)^2=
\begin{cases} -u^2 - 2v + 5 \text{ for }v\in[0,4-2u]\\
(u + v - 3)(3u + v - 7)\text{ for }v\in[4-2u,3-v]
\end{cases}$$\text{ and }$$P(v)\cdot\MC=\begin{cases}
1\text{ for }v\in[0,4-2u]\\
5 - 2u - v\text{ for }v\in[4-2u,3-u]
\end{cases}$$
Thus,
\begin{align*}
    S_D(\MC)=\frac{1}{5-u^2}\Big(\int_0^{4-2u} -u^2 - 2v + 5 dv&+\int_{4-2u}^{3-uu} (u + v - 3)(3u + v - 7) dv\Big)=\\
    &=\frac{4u^3 - 15u^2 + 6u + 17}{15 - 3 u^2}
\end{align*}
Thus, $\delta_P(T,D)\le \frac{15 - 3 u^2}{4u^3 - 15u^2 + 6u + 17}$ for $P\in L_{24}$.
If $P\in L_{24}\backslash (E_2\cup E_4)$ then
$$h_D(v)=\begin{cases}
\frac{1}{2}\text{ for }v\in[0,4-2u]\\
\frac{(5 - 2u - v)^2}{2}\text{ for }v\in[4-2u,3-u]
\end{cases}$$
So for $P\in L_{24}\backslash (E_2\cup E_4)$ we have 
    \begin{align*}
    S_D(W_{\bullet,\bullet}^{\MC};P)=\frac{2}{5-u^2}&\Big(\int_0^{4-2u} \frac{1}{2} dv+\int_{4-2u}^{3-u} \frac{(5 - 2u - v)^2}{2} dv=\\
    &=\frac{u^3 - 6u^2 + 6u + 5}{15 - 3 u^2}\le \frac{4u^3 - 15u^2 + 6u + 17}{15 - 3 u^2}
\end{align*}
We obtain that $$\delta_P(T,D)= \frac{15 - 3 u^2}{u^3 - 6u^2 + 6u + 5}\text{ for }P\in L_{24}\backslash (E_2\cup E_4)\text{ and }u\in[1,2].$$\\
{\bf Step 4.1.} Suppose  $P\in L_{12}\cup L_{14}$ and $u\in [1,3/2]$. In this case we set $\MC= L_{14}$. Then $\tau(\MC)=3-u$. Without loss of generality, we can assume that $P\in L_{14}$. The Zariski Decomposition of the divisor $D-vL_{14}$ is given by:
$$P(v)=
\begin{cases}D-vL_{14}-\frac{v}{2}E_4\text{ for }v\in[0,2-u]\\
D-vL_{14}-\frac{v}{2}E_4-(u+v-2)E_1\text{ for }v\in[2-u,1]\\
D-vL_{14}-\frac{v}{2}E_4-(u+v-2)E_1-(v-1)L_{23}\text{ for }v\in[1,4-2u]\\
D-vL_{14}-(u+v-2)(E_1+E_4)-(v-1)L_{23}-(2u+v-4)E_5\text{ for }v\in[4-2u,3-u]
\end{cases}$$
and 
$$N(v)=
\begin{cases}\frac{v}{2}E_4\text{ for }v\in[0,2-u]\\
\frac{v}{2}E_4+(u+v-2)E_1\text{ for }v\in[2-u,1]\\
\frac{v}{2}E_4+(u+v-2)E_1+(v-1)L_{23}\text{ for }v\in[1,4-2u]\\
(u+v-2)(E_1+E_4)+(v-1)L_{23}+(2u+v-4)E_5\text{ for }v\in[4-2u,3-u]
\end{cases}$$
Moreover
$$P(v)^2=
\begin{cases}5 - 2v - v^2/2 - u^2\text{ for }v\in[0,2-u]\\
9 - 4u - 6v + v^2/2 + 2uv\text{ for }v\in[2-u,1]\\
\frac{(v - 2)(3v + 4u - 10)}{2}\text{ for }v\in[1,4-2u]\\
2( u + v-3)^2\text{ for }v\in[4-2u,3-u]
\end{cases}$$
and
$$P(v)\cdot \MC=
\begin{cases}
v/2 + 1\text{ for }v\in[0,2-u]\\
3 - u - v/2\text{ for }v\in[2-u,1]\\
4 - u - 3v/2\text{ for }v\in[1,4-2u]\\
2(3 - u - v)\text{ for }v\in[4-2u,3-u]
\end{cases}$$
Thus,
\begin{align*}
    S_D(\MC)=\frac{1}{5-u^2}&\Big(\int_0^{2-u} 5 - 2v - v^2/2 - u^2 dv 
    +\int_{2-u}^{1} 9 - 4u - 6v + v^2/2 + 2uv dv+\\
    &+\int_{1}^{4-2u} \frac{(v - 2)(3v + 4u - 10)}{2} dv+\int_{4-2u}^{3-u} 2( u + v-3)^2 dv\Big)=\frac{3 u^3 - 12 u^2 + 6 u + 13}{15 - 3 u^2}
\end{align*}
Thus, $\delta_P(T,D)\le \frac{15 - 3 u^2}{3 u^3 - 12 u^2 + 6 u + 13}$ for $P\in L_{14}$. Note that we have:
\begin{itemize}
\item if $P\in L_{14}\backslash (E_4\cup E_1\cup L_{23}\cup E_5)$ then
$$h_D(v)=\begin{cases}
\frac{(v/2 + 1)^2}{2}\text{ for }v\in[0,2-u]\\
\frac{(3 - u - v/2)^2}{2}\text{ for }v\in[2-u,1]\\
\frac{(4 - u - 3v/2)^2}{2}\text{ for }v\in[1,4-2u]\\
2(3 - u - v)^2\text{ for }v\in[4-2u,3-u]
\end{cases}$$
\item if $P =  L_{14} \cap E_1$ then
$$h_D(v)=\begin{cases}
\frac{(v/2 + 1)^2}{2}\text{ for }v\in[0,2-u]\\
\frac{(6 - 2 u - v ) (2 u + 3 v - 2)}{8}\text{ for }v\in[2-u,1]\\
\frac{(8 -2 u - 3 v) (2 u + v)}{8}\text{ for }v\in[1,4-2u]\\
(3-u - v)\text{ for }v\in[4-2u,3-u]
\end{cases}$$
\item if $P = L_{14}\cap L_{23}$ then
$$h_D(v)=\begin{cases}
\frac{(v/2 + 1)^2}{2}\text{ for }v\in[0,2-u]\\
\frac{(3 - u - v/2)^2}{2}\text{ for }v\in[2-u,1]\\
\frac{(8 -2 u - 3 v) (4 -2 u + v)}{8}\text{ for }v\in[1,4-2u]\\
2 (2 - u) (3 - u - v)\text{ for }v\in[4-2u,3-u]
\end{cases}$$
\end{itemize}
So we have 
\begin{itemize}
    \item if $P\in  L_{14}\backslash (E_4\cup E_1\cup L_{23}\cup E_5)$ then
    \begin{align*}
    S_D(W_{\bullet,\bullet}^{\MC};P)=\frac{2}{5-u^2}&\Big(\int_0^{2-u} \frac{(v/2 + 1)^2}{2} dv+\int_{2-u}^1 \frac{(3 - u - v/2)^2}{2} dv+\\
    &+\int_{1}^{4-2u} \frac{(4 - u - 3v/2)^2}{2} dv+ \int_{4-2u}^{3-u} 2(3 - u - v)^2 dv\Big)=\\
    &=\frac{21-u^3 - 6u}{2(15 - 3 u^2)}\le \frac{3 u^3 - 12 u^2 + 6 u + 13}{15 - 3 u^2}
    \end{align*}
    \item if $P  = L_{14}\cap E_1$ then
    \begin{align*}
    S_D(W_{\bullet,\bullet}^{\MC};P)=\frac{2}{5-u^2}&\Big(\int_0^{2-u} \frac{(v/2 + 1)^2}{2} dv+\int_{2-u}^1 \frac{(6 - 2 u - v ) (2 u + 3 v - 2)}{8} dv+\\
    &+\int_{1}^{4-2u} \frac{(8 -2 u - 3 v) (2 u + v)}{8} dv+ \int_{4-2u}^{3-u}(3-u - v) dv\Big)=\\
    &=\frac{19 - 2u^3}{2(15 - 3 u^2)}
    \end{align*}
        \item if $P  = L_{14}\cap L_{23}$ then
    \begin{align*}
    S_D(W_{\bullet,\bullet}^{\MC};P)=\frac{2}{5-u^2}&\Big(\int_0^{2-u} \frac{(v/2 + 1)^2}{2} dv+\int_{2-u}^1\frac{(3 - u - v/2)^2}{2} dv+\\
    &+\int_{1}^{4-2u} \frac{(8 -2 u - 3 v) (4 -2 u + v)}{8} dv+ \int_{4-2u}^{3-u} 2 (2 - u) (3 - u - v) dv\Big)=\\
    &=\frac{26 - 12u^2 + 3u^3}{2(15 - 3 u^2)}\le \frac{3 u^3 - 12 u^2 + 6 u + 13}{15 - 3 u^2}
    \end{align*}
\end{itemize}
We obtain that  
$$\delta_{P}(T,D)= \frac{15 - 3 u^2}{3 u^3 - 12 u^2 + 6 u + 13}\text{ for }P\in L_{14}\backslash (E_4\cup E_5 \cup E_1)
\text{ and }u\in[1,3/2]$$
and 
$$\delta_{P}(T,D)\ge 
\left\{
\aligned
&\frac{15 - 3 u^2}{3 u^3 - 12 u^2 + 6 u + 13}\text{ for }P=L_{14}\cap E_1 \text{ and }u\in [1,b]\\
&\frac{2(15 - 3 u^2)}{19 - 2u^3}\text{ for }P=L_{14}\cap E_1 \text{ and }u\in [b,3/2]
\endaligned
\right.$$
where $b$ is a root of $8u^3 - 24u^2 + 12u + 7$ on $[1,3/2]$. Note that $b\in [1.261,1.262]$.\\
{\bf Step 4.2.} Suppose $P\in  L_{12}\cup L_{14}$ and $u\in [3/2,2]$. In this case we set $\MC= L_{14}$. Then $\tau(\MC)=3-u$. Without loss of generality, we can assume that $P\in L_{14}$. The Zariski Decomposition of the divisor $D-vL_{14}$ is given by:
$$P(v)=
\begin{cases}D-vL_{14}-\frac{v}{2}E_4\text{ for }v\in[0,2-u]\\
D-vL_{14}-\frac{v}{2}E_4-(u+v-2)E_1\text{ for }v\in[2-u,4-2u]\\
D-vL_{14}-(u+v-2)(E_1+E_4)-(2u+v-4)E_5\text{ for }v\in[4-2u,1]\\
D-vL_{14}-(u+v-2)(E_1+E_4)-(v-1)L_{23}-(2u+v-4)E_5\text{ for }v\in[1, 3-u]
\end{cases}$$
and 
$$N(v)=
\begin{cases}
\frac{v}{2}E_4\text{ for }v\in[0,2-u]\\
\frac{v}{2}E_4+(u+v-2)E_1\text{ for }v\in[2-u,4-2u]\\
(u+v-2)(E_1+E_4)+(2u+v-4)E_5\text{ for }v\in[4-2u,1]\\
(u+v-2)(E_1+E_4)+(v-1)L_{23}+(2u+v-4)E_5\text{ for }v\in[1, 3-u]
\end{cases}$$
Moreover
$$P(v)^2=
\begin{cases}5 - 2v - v^2/2 - u^2\text{ for }v\in[0,2-u]\\
9 - 4u - 6v + v^2/2 + 2uv\text{ for }v\in[2-u,4-2u]\\
2u^2 + 4uv + v^2 - 12u - 10v + 17\text{ for }v\in[4-2u,1]\\
2(u + v-3)^2\text{ for }v\in[1,3-u]
\end{cases}$$
and
$$P(v)\cdot\MC=
\begin{cases}
1 + v/2\text{ for }v\in[0,2-u]\\
3 - u - v/2\text{ for }v\in[2-u,4-2u]\\
5 - 2u - v\text{ for }v\in[4-2u,1]\\
2(3 - u - v)\text{ for }v\in[1,3-u]
\end{cases}$$
Thus,
\begin{align*}
    S_D(\MC)=\frac{1}{5-u^2}\Big(\int_0^{2-u} 5 - 2v - v^2/2 - u^2 dv 
    &+\int_{2-u}^{4-2u} 9 - 4u - 6v + v^2/2 + 2uv dv+\\
    &+\int_{4-2u}^1 2u^2 + 4uv + v^2 - 12u - 10v + 17 dv+\\
    &+\int_{1}^{3-u} 2( u + v-3)^2 dv\Big)=\frac{3u^3 - 12u^2 + 6u + 13}{15 - 3 u^2}
\end{align*}
Thus, $\delta_P(T,D)\le \frac{15 - 3 u^2}{3u^3 - 12u^2 + 6u + 13}$ for $P\in L_{14}$. Note that we have:
\begin{itemize}
\item if $P\in L_{14}\backslash (E_4\cup E_1\cup L_{23}\cup E_5)$ then
$$h_D(v)=\begin{cases}
\frac{(1+ v/2)^2}{2}\text{ for }v\in[0,2-u]\\
\frac{(3 - u - v/2)^2}{2}\text{ for }v\in[2-u,4-2u]\\
\frac{(5 - 2u - v)^2}{2}\text{ for }v\in[4-2u,1]\\
2(3 - u - v)^2\text{ for }v\in[1,3-u]
\end{cases}$$
\item if $P = L_{14} \cap E_1$ then
$$h_D(v)=\begin{cases}
\frac{(1+ v/2)^2}{2}\text{ for }v\in[0,2-u]\\
\frac{ (6 -2 u - v ) (2 u + 3 v - 2)}{8}\text{ for }v\in[2-u,4-2u]\\
\frac{(v + 1) (5 -2 u - v)}{2}\text{ for }v\in[4-2u,1]\\
2 (3-u - v)\text{ for }v\in[1,3-u]
\end{cases}$$
\item if $P= L_{14}\cap L_{23}$ then
$$h_D(v)=\begin{cases}
\frac{(1+ v/2)^2}{2}\text{ for }v\in[0,2-u]\\
\frac{(3 - u - v/2)^2}{2}\text{ for }v\in[2-u,4-2u]\\
\frac{(5 - 2u - v)^2}{2}\text{ for }v\in[4-2u,1]\\
2 (2 - u) (3 -u - v)\text{ for }v\in[1,3-u]
\end{cases}$$
\end{itemize}
\begin{itemize}
    \item if $P\in  L_{14}\backslash (E_4\cup E_1\cup L_{23}\cup E_5)$ then
    \begin{align*}
    S_D(W_{\bullet,\bullet}^{\MC};P)=\frac{2}{5-u^2}&\Big(\int_0^{2-u} \frac{(v/2 + 1)^2}{2} dv+\int_{2-u}^{4-2u} \frac{(3 - u - v/2)^2}{2} dv+\\
    &+\int_{4-2u}^1 \frac{(5-2u-v)^2}{2} dv+ \int_{1}^{3-u} 2(3 - u - v)^2 dv\Big)=\\
    &=\frac{7u^3 - 36u^2 + 48u - 6}{2(15 - 3 u^2)}\le \frac{3 u^3 - 12 u^2 + 6 u + 13}{15 - 3 u^2}
    \end{align*}
    \item if $P = L_{14}\cap E_1$ then
    \begin{align*}
    S_D(W_{\bullet,\bullet}^{\MC};P)=\frac{2}{5-u^2}&\Big(\int_0^{2-u} \frac{(v/2 + 1)^2}{2} dv+\int_{2-u}^{4-2u} \frac{ (6 -2 u - v ) (2 u + 3 v - 2)}{8} dv+\\
    &+\int_{4-2u}^1 \frac{(v + 1) (5 -2 u - v)}{2} dv+ \int_{1}^{3-u} 2 (3-u - v) dv\Big)=\\
    &=\frac{3u^3 - 18u^2 + 27u - 4}{15 - 3 u^2}
    \end{align*}
    \item if $P = L_{14}\cap L_{23}$ then
    \begin{align*}
    S_D(W_{\bullet,\bullet}^{\MC};P)=\frac{2}{5-u^2}&\Big(\int_0^{2-u} \frac{(v/2 + 1)^2}{2} dv+\int_{2-u}^{4-2u} \frac{(3 - u - v/2)^2}{2} dv+\\
    &+\int_{4-2u}^1 \frac{(5-2u-v)^2}{2} dv+ \int_{1}^{3-u} 2 (2 - u) (3 -u - v) dv\Big)=\\
    &=\frac{3u^3 - 12u^2 + 26}{2(15 - 3 u^2)}\le \frac{3 u^3 - 12 u^2 + 6 u + 13}{15 - 3 u^2}
    \end{align*}
\end{itemize}
We obtain that 
$$\delta_{P}(T,D)= \frac{15 - 3 u^2}{3 u^3 - 12 u^2 + 6 u + 13}\text{ for }P\in L_{14}\backslash (E_1\cup E_4 \cup E_5) \text{ and }u\in[3/2,2]$$
and 
$$\delta_{P}(T,D)\ge \frac{15 - 3 u^2}{3u^3 - 18u^2 + 27u - 4}\text{ for }P=L_{14}\cap E_1 \text{ and }u\in [3/2,2]$$
\end{proof}
\begin{corollary}\label{cor2A1}
Let $P$ be a point in $T$ that is contained in $ L_{12}\cup L_{14}\cup L_{24}\cup E_2\cup E_3\cup E_4\cup E_4$ then
$$\delta_{P}(T,D)\ge
\left\{
\aligned
&\frac{15 - 3 u^2}{16 + 3 u - 9 u^2 + 2 u^3}\text{ for }u\in [1,a],\\
&\frac{15 - 3 u^2}{11 - u^3}\text{ for }u\in [a,2]
\endaligned
\right.$$
\end{corollary}
\begin{corollary}\label{cor2A1sing}
Suppose $O$ is a point on a del Pezzo surface $\overline{T}$ with two $\mathbb{A}_1$ singularities and nine lines such that $\delta_O(T)\le \frac{6}{5}$ then
$$\delta_{O}(\overline{T},\overline{D})\ge
\left\{
\aligned
&\frac{15 - 3 u^2}{16 + 3 u - 9 u^2 + 2 u^3}\text{ for }u\in [1,a],\\
&\frac{15 - 3 u^2}{11 - u^3}\text{ for }u\in [a,2]
\endaligned
\right.$$
\end{corollary}
\subsection{Polarized $\delta$-invariant on Del Pezzo surface of degree $4$ with $\DA_2$ singularity} \label{A2}
Now, let us use the notations and assumptions of Section 2 with a minor difference: we assume that $\overline{T}$ has a singular point of type $\mathbb{A}_2$. Let us show that in the case when $O$ is the singular point of the surface $\overline{T}$ we have
$$
\delta_O(\overline{T},\overline{D})=\frac{u^3 - 6u^2 + 19}{15-3u^2}
$$
 which immediately implies that $\delta_O(\overline{T},W^{\overline{T}}_{\bullet,\bullet})\le \frac{80}{83}$. In this case, the morphism $\eta$ is a blow up of $\DP^2$ at points
 $P_1$, $P_2$, and $P_3$ in general position; after that blowing up a point $P_4$ which belongs to the exceptional divisor corresponding to $P_3$ and no other negative curve and after that a point $P_5$ which belongs to the exceptional divisor corresponding to $P_4$ and no other negative curve. By \cite[Section 6.5]{Denisova} we have:
$$\delta_P(T)=
\left\{
\aligned
&6/7\text{ if }P\in E_3\cup E_4,\\
&8/7\text{ if }P\in (L_{13}\cup L_{23}\cup L_{34}\cup E_5)\backslash (E_3\cup E_4),\\
& 4/3 \text{ if }P\in (L_{12}\cup C_2)\cap (E_1\cup E_2),\\
&  18/13\text{ if }P\in (L_{12}\cup C_2\cup E_1\cup E_2)\backslash ((L_{12}\cup C_2)\cap (E_1\cup E_2)),\\
& 3/2,\text{ otherwise }
\endaligned
\right.$$
where $E_1$, $E_2$, $E_3$, $E_4$, $E_5$ are exceptional divisors corresponding to $P_1$, $P_2$, $P_3$, $P_4$, $P_5$ respectively, $C_{2}$ is a strict transform of a $(-1)$-curve coming from the conic on $\DP^2$, $L_{ij}$ are strict transforms of the lines passing through $P_i$ and $P_j$ for $(i,j)\in\{(1,2),(1,3),(2,3)\}$ and $L_{34}$ is a strict transform of a $(-1)$-curve coming from a line passing through $P_3$  on $\DP^2$. The dual graph of $(-1)$ and $(-2)$-curves is given in the following picture:
       \begin{center}
      \includegraphics[width=12cm]{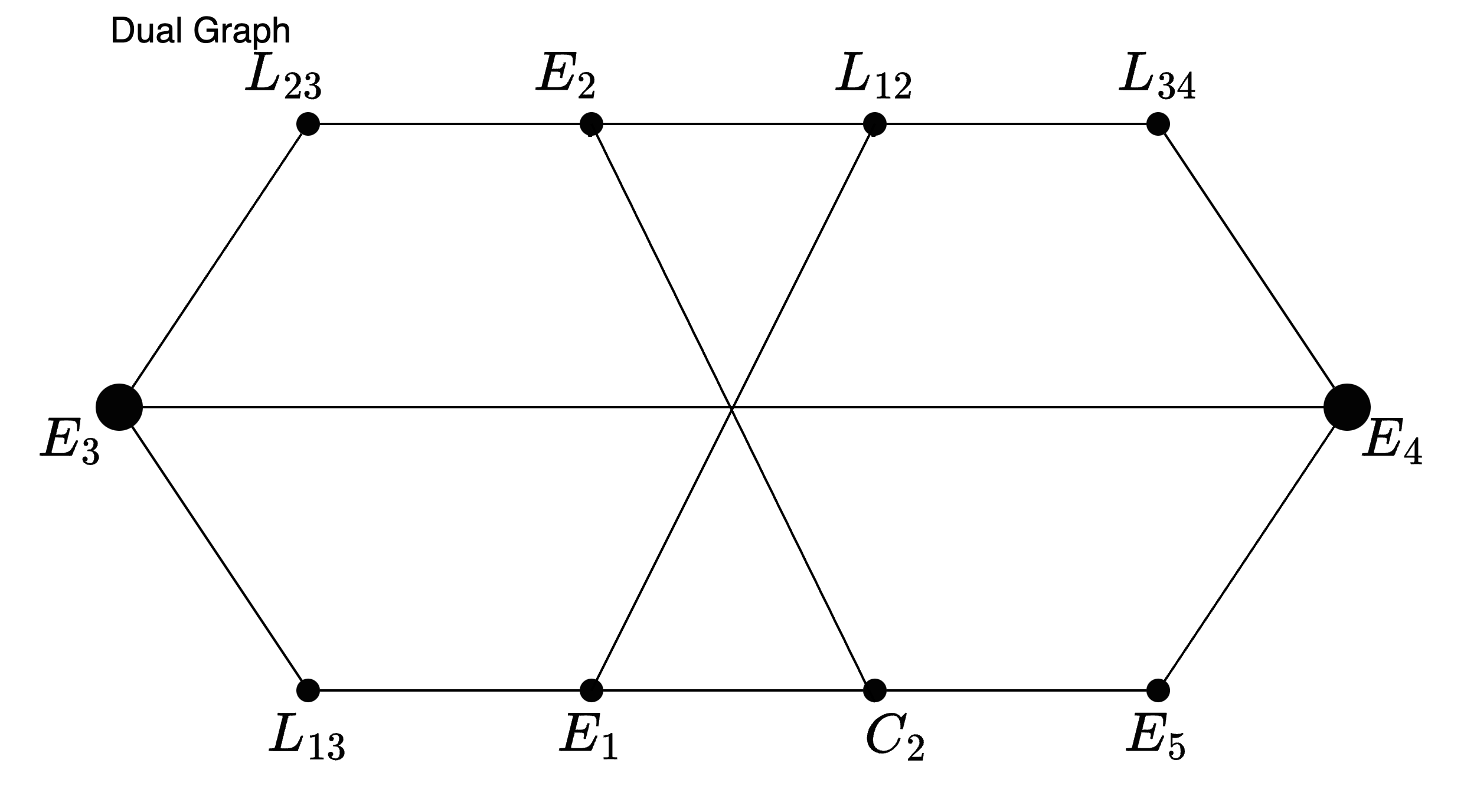}
   \end{center}
Now let's prove that:
$$\delta_{P}(T,D)= \frac{u^3 - 6u^2 + 19}{15-3u^2}\text{ for }P\in E_4\backslash (L_{34}\cup E_5)$$
\begin{proof}
Suppose $P\in E_4\backslash (L_{34}\cup E_5)$.  In this case we set $\MC= E_4$. Then $\tau(\MC)=2$. The Zariski Decomposition of the divisor $D-vE_4$ is given by:
$$P(v)=
\begin{cases}-K_{T}-(u-1)C_2-vE_4-\frac{v}{2}E_3\text{ for }v\in[0,2-u]\\
-K_{T}-(u-1)C_2-vE_4-\frac{v}{2}E_3-(u+v-2)E_5\text{ for }v\in[2-u,1]\\
-K_{T}-(u-1)C_2-vE_4-\frac{v}{2}E_3-(u+v-2)E_5-(v-1)L_{34}\text{ for }v\in[1,2]
\end{cases}$$
and 
$$N(v)=
\begin{cases}
\frac{v}{2}E_3\text{ for }v\in[0,2-u]\\
\frac{v}{2}E_3+(u+v-2)E_5\text{ for }v\in[2-u,1]\\
\frac{v}{2}E_3+(u+v-2)E_5+(v-1)L_{34}\text{ for }v\in[1,2]
\end{cases}$$
Moreover, 
$$P(v)^2=\begin{cases}5 -u^2 - 3v^2/2 \text{ for }v\in[0,2-u]\\
9 - 4u - 4v + 2uv - 1/2v^2\text{ for }v\in[2-u,1]\\
\frac{(v - 2)(v + 4u - 10)}{2}\text{ for }v\in[1,2]
\end{cases}$$\text{ and }$$P(v)\cdot \MC=\begin{cases}
3v/2\text{ for }v\in[0,2-u]\\
2-u+v/2\text{ for }v\in[2-u,1]\\
3-u-v/2\text{ for }v\in[1,2]
\end{cases}$$
Thus,
\begin{align*}
    S_D(\MC)=\frac{1}{5-u^2}\Big(\int_0^{2-u} 5 -u^2 - 3v^2/2 dv&+\int_{2-u}^1 9 - 4u - 4v + 2uv - 1/2v^2 dv+\\
    &+\int_{1}^{2} \frac{(v - 2)(v + 4u - 10)}{2} dv\Big)=\frac{19+u^3 - 6u^2}{15 - 3 u^2}
\end{align*}
Thus, $\delta_P(T,D)\le \frac{15 - 3 u^2}{19+u^3 - 6u^2}$ for $P\in E_4$. Note that for $P\in E_4\backslash (E_5\cup L_{34})$ we have:
$$h_D(v)=\begin{cases}
\frac{9v^2}{8}\text{ for }v\in[0,2-u]\\
\frac{(2-u+v/2)^2}{2}\text{ for }v\in[2-u,1]\\
\frac{(3-u-v/2)^2}{2}\text{ for }v\in[1,2]
\end{cases}$$
So we have
    \begin{align*}
    S_D(W_{\bullet,\bullet}^{\MC};P)=\frac{2}{5-u^2}\Big(\int_0^{2-u} \frac{9v^2}{8} dv&+\int_{2-u}^1 \frac{(2-u+v/2)^2}{2} dv+\int_{1}^{2} \frac{(3-u-v/2)^2}{2} dv\Big)=\\
    &=\frac{21 + 6 u - 18 u^2 + 5 u^3}{2(15 - 3 u^2)}\le \frac{19+u^3 - 6u^2}{15 - 3 u^2}
\end{align*}
So we obtain that
$$\delta_{P}(T,D)= \frac{u^3 - 6u^2 + 19}{15-3u^2}\text{ for }P\in E_4\backslash (L_{34}\cup E_5).$$
\end{proof}

\end{document}